\documentclass[11pt]{article}

\usepackage{latexsym}
\usepackage{amsthm}
\usepackage{amsmath}
\usepackage{amsfonts, epsfig, amsmath, amssymb, color, amscd}
\usepackage[cmtip,arrow]{xy}
\usepackage{pb-diagram, pb-xy}
\usepackage{amssymb,epsfig}
\usepackage{color}
\usepackage[cmtip,arrow]{xy}
\usepackage{pb-diagram, pb-xy}
\usepackage{amssymb,epsfig,amsfonts}

\newfont{\sdbl}{msbm9}
\newfont{\dbl}{msbm10 at 12pt}
\theoremstyle{definition}
\newcommand{\da}{{\mbox{\dbl A}}}
\newcommand{\dw}{{\mbox{\dbl W}}}
\newcommand{\cn}{{{\cal N}}}
\newcommand{\ca}{{{\cal A}}}
\newcommand{\cl}{{{\cal L}}}
\newcommand{\cg}{{{\cal G}}}
\newcommand{\dpp}{{\mbox{\dbl P}}}
\newcommand{\dz}{{\mbox{\dbl Z}}}

\newcommand{\dn}{{\mbox{\dbl N}}}

\newcommand{\sdn}{{\mbox{\sdbl N}}}
\newcommand{\sdp}{{\mbox{\sdbl P}}}
\newcommand{\dc}{{\mbox{\dbl C}}}
\newcommand{\dq}{{\mbox{\dbl Q}}}

\newcommand{\ord}{\mathop{\rm ord}\nolimits}

\newcommand{\Pic}{\mathop {\rm Pic}}

\newcommand{\Adm}{\mathop {\rm Adm}}
\newcommand{\Div}{\mathop {\rm Div}}
\newcommand{\WDiv}{\mathop {\rm WDiv}}

\newcommand{\LT}{\mathop {\rm LT}}

\newcommand{\Sup}{\mathop {\rm Supp}}
\newcommand{\Quot}{\mathop {\rm Quot}}

\newcommand{\rk}{\mathop {\rm rk}}

\newcommand{\Spec}{\mathop {\rm Spec}}

\newcommand{\Proj}{\mathop {\rm Proj}}

\newcommand\limind{\mathop{\underrightarrow{\lim}}}

\makeatother
\newtheorem{defin}{Definition}[section]
\newtheorem{nt}{Remark}[section]

\newtheorem{ex}{Example}[section]

\theoremstyle{plain}
\newtheorem{prop}{Proposition}[section]
\newtheorem{theo}{Theorem}[section]
\newtheorem{lemma}{Lemma}[section]
\newtheorem{corol}{Corollary}[section]

\newcommand{\eqdef}{\stackrel{\rm def}{=}}

\newcommand{\lto}{\longrightarrow}
\newcommand{\Ord}{\mathop {\rm \bf ord}}
\newcommand{\Rk}{\mathop {\rm \bf rk}}
\newcommand{\co}{{{\cal O}}}
\newcommand{\cf}{{{\cal F}}}
\newcommand{\cq}{{{\cal Q}}}

\newcommand{\cm}{{{\cal M}}}
\newcommand{\cw}{{{\cal W}}}

\newcommand{\cb}{{{\cal B}}}

\newcommand{\cp}{{{\cal P}}}

\title{Geometric properties of commutative subalgebras of partial differential operators}
\author{Herbert Kurke \quad Alexander Zheglov}
\date{}

\begin{document}

\maketitle

\begin{abstract}
We investigate further alebro-geometric properties of commutative rings of partial differential operators continuing our  research started in previous articles. In particular, we start to explore the most evident examples and also 
certain known examples of algebraically integrable quantum completely integrable systems from the point of view of a recent generalization of Sato's theory which belongs to the second author. We give a complete characterisation of the spectral data for a class of "trivial" rings and strengthen geometric properties known earlier for a class of known examples. We also define a kind of a restriction map from the moduli space of coherent sheaves with fixed Hilbert polynomial on a surface to analogous moduli space on a divisor (both the surface and divisor are part of the spectral data). We give several explicit examples of spectral data and corresponding rings of commuting (completed) operators, producing as a by-product interesting examples of surfaces that are not isomorphic to spectral surfaces of any (maximal) commutative ring of PDOs of rank one. At last, we prove that any commutative ring of PDOs, whose normalisation is isomorphic to the ring of polynomials $k[u,t]$, is a Darboux transformation of a ring of operators with constant coefficients.  

\end{abstract}


\section{Introduction}

\subsection{}\label{s1.1}

In this paper we continue the study of algebro-geometric properties of commutative algebras of partial differential operators (PDO for short) in two variables started in \cite{Ku3}. Everywhere in this paper we assume that $k$ is a field of characteristic zero.

Recall that one of very complicated questions appearing in the theory of algebraically integrable systems is: how to find explicit examples of certain commutative rings of PDOs or how to classify them (see \cite[Introduction]{Ku3} for an extensive history). This question can also be reformulated in the following way. In \cite{BEG} the quantum analogue of the classical definition of an integrable Hamiltonian system was defined. By a quantum completely integrable system (QCIS) on an algebraic variety $X$ the authors understand a pair $(\Lambda ,\theta )$, where $\Lambda$ is an irreducible $n$-dimensional affine algebraic variety, and $\theta :\co_{\Lambda}\rightarrow D(X)$ is an embedding of algebras (here the algebra $D(X)$ of differential operators on $X$ is the quantum analogue of the Poisson algebra $\co (T^*X)$). 

By definition, a QCIS $S=(\Lambda ,\theta )$ is said to be algebraically integrable if it is dominated by another QCIS $S'$ with $rk (S')=1$ (see loc. cit.), where the rank of QCIS is the dimension of the space of formal solutions of the system 
$$
\theta (g)\psi =g(\lambda )\psi , \mbox{\quad} g\in \co_{\Lambda }
$$
near a generic point of $X$. In \cite{BEG} these definitions were also generalized to the case of integrable systems on a formal polydisc. Thus, in this case $X$ is $\Spec (k[[x_1,x_2,\ldots ,x_n]])$ and the symbols $\co_X, k(X),D(X)$ denote respectively $k[[x_1,\ldots ,x_n]]$, $k((x_1,\ldots ,x_n))$, $\co_X[\partial_1,\ldots ,\partial_n]$, where $\partial_i=\partial /\partial x_i$. In this situation for $n=1$ even the classification of all algebraically integrable commutative subalgebras $B=\theta (\Lambda )\subset D(X)$ in terms of the spectral data is known since the work of Krichever \cite{Kr1}, \cite{Kr}. In \cite{BEG} the criterion for algebraic integrability of QCIS's is given in terms of the corresponding Galois groups. 

\subsection{}\label{s1.2}
In this paper we continue to explore geometric properties of commutative rings of PDOs in $D=k[[x_1,x_2]][\partial_1, \partial_2]$ started in \cite{Ku3} (the restriction $n=2$ seems to be basically not essential, but in general case one needs to do some work to generalize a number of statements from our previous papers). Recall that even in this case there is still no classification of algebraically integrable (in the above sense) commutative subalgebras in terms of spectral data, though there is a classification of subalgebras in a completed ring of differential operators (see \cite{Zhe2}, cf. \cite[Introduction]{Ku3}) in terms of Parshin's modified geometric data (which include algebraic projective surface, an ample $\dq$-Cartier divisor, a point regular on this divisor and on the surface, a torsion free sheaf on the surface and some extra trivialisation data). Moreover, up to now  only a few examples of such algebras are known. Probably the first nontrivial (in certain sense, see discussion below) examples  appeared in \cite{ChV}, \cite{ChV2}, \cite{ChV3}. The examples were connected  with the quantum (deformed) Calogero-Moser systems (cf.  \cite{OP}).
Later the ideas of these constructions  were developed in a series of papers (see e.g. \cite{FV}, \cite{FV2}, \cite{EG}) in order to construct more examples (for review see e.g. \cite{Ch} and references therein; cf. also \cite{BEG}, \cite{BEGa}, \cite{BeK}). Let's also mention that the idea to construct a free BA-module (the module consisting of eigenfunctions of the ring of PDO) was developed later by various authors (see e.g. \cite{Na}, \cite{Mi}, \cite{Ch}) to produce explicit examples of commutative matrix rings of PDO.

 In \cite{Zhe2}, \cite{Ku3} several properties of the above mentioned geometric data were investigated. In particular, all algebraically integrable commutative rings of PDOs correspond to rank one geometric data with $X$ Cohen-Macaulay, $C$ rational and $C^2=1$, $\cf$ torsion free of rank one and Cohen-Macaulay along $C$. In this paper we strengthen  the last property: namely, we show that any commutative subalgebra of PDOs (satisfying as in \cite{Ku3} certain mild conditions) leads to a sheaf $\cf$ on $X$ which is Cohen-Macaulay (theorem \ref{CMmodules}). 
 
 Cohen-Macaulay rank one torsion free sheaves appearing as sheaves from geometric data classifying commutative subalgebras of (completed) operators with fixed spectral surface can be parametrized by a moduli space which is an open subscheme of the projective scheme parametrising semistable sheaves with fixed Hilbert polynomial (see remark \ref{moduli_space}). We introduce in this paper a kind of restriction map $\zeta$ from this moduli space to the moduli space of coherent torsion free rank one sheaves on the divisor $C$ (see section \ref{mapzeta}, remark \ref{moduli_space}) and formulate a conjecture that this morphism is surjective (remark \ref{moduli_space}). It is important to study this moduli space in order to find new examples of algebraically integrable systems or to classify commutative algebras of PDOs. We hope to return to this question in future works. 
 
 This moduli space can be thought of as another analogue of the Jacobian of the curve in the context of the classical KP theory. Recall that in the work \cite{Pa0} Parshin offered to consider a multi-variable analogue of the KP-hierarchy which, being modified, is related to algebraic surfaces and torsion free sheaves on such surfaces as well as to a wider class of geometric data consisting of ribbons and torsion free sheaves on them if the number of variables is equal to two (see \cite{Zhe},  \cite[Introduction]{Ku1}). In the work \cite{Ku1} we described the geometric structure of the Picard scheme of a ribbon. This scheme has a nice group structure and can be thought of as an analogue of the Jacobian of a curve in the context of the classical KP theory. In particular, generalized  KP flows are defined on such schemes (flows defined by the multi-variable analogue of the KP-hierarchy). The disadvantage of the Picard scheme of a ribbon is its infinite-dimensionality. The moduli space we have mentioned above is finite dimensional. The generalized  KP flows are also defined on it. It is not difficult to show that it can be embedded into the Picard scheme of a ribbon. 
 
 Investigating already existing examples of commutative algebras mentioned above we prove a theorem (\ref{Darboux}) about algebraically integrable commutative rings of PDOs whose affine spectral surface is rational. Such rings appeared, for example, in papers \cite{FV}, \cite{FV2}, \cite{EG}, \cite{BeK}. In the examples from these papers the normalisation of the affine spectral surface is known to be $\da^2$. 
 In \cite{BeK} the authors gave a method of producing new non-trivial examples of commutative rings of PDOs using the Darboux transformation. We show in theorem \ref{Darboux} that all rings with this property of the affine spectral surface are Darboux transformations of rings of operators with constant coefficients. As a by-product we also give a geometric characterisation of certain completion of $\da^2$ (see theorem \ref{completion_of_plane}): a completion of $\da^2$, whose divisor at infinity is an ample irreducible $\dq$-Cartier divisor with self-intersection   index 1, is $\dpp^2$. This result could be probably proved by classical methods of algebraic geometry using old results of Morrow (\cite{Mo}) or relatively new results of Kojima, Takahashi (\cite{Ko}) (we would like to thank M.Gizatoulline and T.Bandman for pointing out these works),  but we used instead only some ideas from our theory of ribbons and (or alternately) the construction of the generalized Krichever-Parshin map.
 
 It is reasonable to ask if there are examples of algebraically integrable commutative rings of PDOs whose spectral surface is isomorphic to a given one. We give here two counterexamples (\ref{counterex_aff}, \ref{ex1}), both for affine and projective spectral surfaces. 

Another natural question is: how to characterise those commutative algebras which consist of operators not depending on $x_1$ or $x_2$. We call these algebras "trivial", because one can easily construct such algebras taking commutative subalgebras of one-variable operators and adding a derivation with respect to another variable. Surprisingly the geometry of spectral data  is not so trivial for these algebras. We give a description (theorem \ref{trivial}) of such algebras in terms of geometric data. 

At last, we give examples (\ref{ex1}, \ref{ex2}) of surfaces for which it is possible to describe all sheaves from the moduli space mentioned above and calculate all corresponding rings of commuting (completed) operators. All these rings are "trivial". 

\subsection{}\label{s1.3}

The paper is organized as follows:

In section \ref{data} we recall basic definition of geometric data from \cite{Zhe2} and give also an alternative definition of these data. 

In section \ref{pairs} we recall the construction of Schur pairs associated with geometric data. 

In section \ref{mapzeta} we introduce the restriction map $\zeta$ and prove several technical lemmas.

In section \ref{CRO} we recall basic definitions and properties of the ring of completed operators, recall the classification theorem from \cite{Zhe2} and prove additional technical lemmas needed in the rest of the paper.

In section \ref{sec.2.3} we recall and prove some properties of Schur pairs corresponding to geometric data with sheaves whose Hilbert polynomial is fixed. For understanding the map $\zeta$ the most important properties are formulated in proposition \ref{zero_cohomology}. We also formulate a conjecture about the map $\zeta$. 

In section 3 we prove theorems about CM-property, completion of plane and Darboux transformations mentioned above. 

In section 4 we give the description of "trivial" algebras and examples.  

{\bf Acknowledgements.} Part of this research was done at the Humboldt University of Berlin during a stay supported by the Vladimir Vernadskij stipendium of MSU-DAAD (Referat 325-paw, Kennziffer A/12/89240). Part of this research was done at the Max Planck Institute of Mathematics during a research stay of the second author (1-30 September 2013). 
He would like to thank the MPI for the excellent working conditions. 

The first author would like to thank the Department of Differential Geometry and Applications of Moscow State University for hospitality during his stay in Moscow in May 2012, where some aspects of this work were planned.  

We are also grateful to Igor Burban, Antonio Laface and Denis Osipov for their permanent interest and many stimulating discussions.

The second author was partially supported by the RFBR grant no.~14-01-00178-a, 13-01-00664à and by grant NSh no.~581.2014.1.

\section{Preliminaries}

In this work we usually use standard notation from algebraic geometry used e.g. in the book \cite{Ha}. We also use some notation from our previous papers \cite{Zhe2}, \cite{Ku3}.  

On the two-dimensional local field $k((u))((t))$ we will consider the following discrete valuation of rank two
$\nu \, : \, k((u))((t))^* \to \dz \oplus \dz $:
$$
\nu (f) = (m,l)  \quad \mbox{iff} \quad f= t^lu^m f_0 \mbox{, where} \quad f_0 \in k[[u]]^*+t k((u))[[t]] \mbox{.}
$$
(Here $k[[u]]^*$ means the set of invertible elements in the ring $k[[u]]$.) We also define the discrete valuation of rank one
$$
\nu_t(f)=l.
$$

\subsection{Geometric data}
\label{data}

 In this subsection we recall definitions from \cite{Zhe2}, \cite{Ku3}; we slightly change some definitions from loc.cit. to simplify the exposition and to avoid explaining certain technical details. 

For any $n$-dimensional irreducible projective variety $X$ over the field $k$, and any Cartier divisors $E_1, \ldots, E_n \in \Div(X)$ on $X$ one defines the intersection index $(E_1 \cdot \ldots \cdot E_n)  \in \dz$ on $X$ (see, e.g.,~\cite{Fu}, \cite[ch.~1.1]{La}.)
Let $(E^n)= (E \cdot \ldots \cdot E)$ be the self-intersection index of a Cartier divisor $E \in \Div(X)$ on $X$, and $\cf$ be a coherent sheaf
on $X$. There is the asymptotic Riemann-Roch theorem (see survey in~\cite[ch.~1.1.D]{La}) which says that  the Euler characteristic
$\chi(X, \cf \otimes_{\co_X} \co_X(mE))$ is a polynomial of degree $\le n$ in $m$, with
\begin{equation}  \label{rrf}
\chi(X, \cf \otimes_{\co_X} \co_X(mE)) = \rk(\cf) \cdot \frac{(E^n)}{n!} \cdot m^n + O(m^{n-1}) \mbox{,}
\end{equation}
where $\rk$ is the rank of the sheaf.

There is the cycle map: ${\rm Z} : \Div(X) \to \WDiv(X)$ from the Cartier divisors to the Weil divisors on $X$ (see \cite[Appendix A]{Ku3}). If $E_1, E_2 \in \Div(X)$ such that ${\rm Z}(E_1) = {\rm Z}(E_2)$, then the self-intersection indices $(E_1^n)= (E_2^n)$ on $X$ (see \cite[\S 2.4]{Ku3}). 

The cycle map $\rm Z$ restricted to the semigroup of effective Cartier divisors $\Div^+(X)$ is an injective map to the semigroup of effective Weil divisors $\WDiv^+(X)$ not contained in the singular locus. We {\em will say } that  an effective Weil divisor $C$ on $X$, not contained in the singular locus, is a {\em $\dq$-Cartier} divisor on $X$ if $lC \in { \rm Im} \, ({\rm Z} \mid_{\Div^+(X)})$ for some integer $l >0$.

\begin{defin}
Let $C$ be a $\dq$-Cartier divisor on $X$. We define the self-intersection index $(C^n)$ on $X$ as
\begin{equation} \label{intin}
(C^n) = (G^n)/ l^n \mbox{,}
\end{equation}
where $G= lC$ is a Cartier divisor for some integer $l >0 $.
\end{defin}
We note that if $l >0$ is  minimal  such that $lC$ is a Cartier divisor, then for any other $l' > 0 $ with the property that $l'C$ is a Cartier divisor
we have that  $l \mid l'$. Therefore, using the property  $(E_1^n)= m^n(E_2^n)$ for any $E_1= m E_2$, $ E_2 \in \Div(X)$, $m \in \dz$
we obtain that formula~\eqref{intin} does not depend on the choice of appropriate $l$.

\smallskip

\begin{defin}
\label{geomdata}
We call $(X,C,P,\cf ,\pi , \phi )$ a geometric data of rank $r$ if it consists of the following data (where we fix the ring $k[[u,t]]$ for all data):
\begin{enumerate}
\item\label{dat1}
$X$ is a reduced irreducible projective algebraic surface defined over a field $k$;
\item\label{dat2}
$C$ is a reduced irreducible ample $\dq$-Cartier  divisor on $X$;
\item\label{dat3}
$P\in C$ is a closed $k$-point, which is
regular on $C$ and on $X$;
\item\label{dat4}
$$
\pi : \widehat{\co}_{P}\longrightarrow k[[u,t]]
$$
is a local $k$-algebra homomorphism satisfying the following
property. If $f$ is a local equation of the curve $C$ at $P$, then
$\pi (f)k[[u,t]] = t^rk[[u,t]]$ and the induced map $\pi :
\hat{\co}_{C,P}=\hat{\co}_P/(f) \rightarrow k[[u]] = k[[u,t]]/(t)$ is an
isomorphism. (The definition of $\pi$ does not depend on the choice
of appropriate $f$. Besides, from this definition it follows that
$\pi$ is an embedding, $k[[u,t]]$ is a free $\widehat{\co}_P$-module
of rank $r$ with respect to $\pi$. Moreover for any element $g$ from
the maximal ideal $\cm_P$ of $\co_P$ such that elements $g$ and $f$
generate $\cm_P$ we obtain that $\nu(\pi(f))=(0,r)$,
$\nu(\pi(g))=(1,0)$.)

\item\label{dat5}
$\cf$ is a torsion free quasi-coherent sheaf on $X$.
\item\label{dat6}
$\phi :{\cf}_P \hookrightarrow k[[u,t]]$ is an ${\co}_P$-module embedding
subject to the following condition for any $n \ge 0$ (we note that by item~\ref{dat4} of this definition, $k[[u,t]]$ is an ${\co}_P$-module with respect to $\pi$).
By item~\ref{dat2} there is the minimal natural number $d$ such that $C'=dC$ is a very ample divisor on $X$. Let $\gamma_n : H^0(X, \cf (nC'))\hookrightarrow {\cf}(nC')_P$ be an embedding (which is an embedding, since $\cf (nC')$ is a torsion free quasi-coherent sheaf on $X$).
Let $\epsilon_n : {\cf}(nC')_P \to \cf_P$ be the natural ${\co}_P$-module isomorphism  given by multiplication to an element $f^{nd} \in {\co}_{P}$, where $f \in {\co}_{P}$ is chosen as in item~\ref{dat4}. Let $\tau_n : k[[u,t]] \rightarrow k[[u,t]]/(u,t)^{ndr+1}$ be the natural
ring epimorphism. We demand that the map
$$    \tau_n \circ \phi \circ  \epsilon_n \circ \gamma_n \, : \, H^0(X, \cf (nC'))   \lto  k[[u,t]]/(u,t)^{ndr+1}$$
is an isomorphism. (These conditions on the map $\phi$ do not depend on the choice of the appropriate element $f$.)
\end{enumerate}
\end{defin}

\begin{nt}
\label{ophi}
The rank of the sheaf is greater or equal to the rank of the data, cf. \cite[Rem.3.3]{Ku3}. 
If the sheaf $\cf$ is coherent of rank one, then $\pi$ is an isomorphism and $\phi$ induces the isomorphism $\hat{\phi}:\hat{\cf}_P\simeq k[[u,t]]$, see \cite[Rem. 3.7]{Zhe2}. Note that any two trivialisations $\hat{\phi}_1, \hat{\phi}_2: \hat{\cf}_P\simeq k[[u,t]]$ differ by multiplication on an element $a\in k[[u,t]]^*$. In some cases 
the conditions on the map $\phi$ in last item of the definition can be rewritten in purely algebro-geometrical terms, see proposition \ref{zero_cohomology} below.
\end{nt}

\subsubsection{Alternative definition of geometric data}

In this section we would like to give an alternative definition of the geometric data. This definition seems to be more "geometric".

Let's introduce the following notation: $T=\Spec k[[u,t]]\supset T_1=\Spec k[[u]]$
(defined by $t=0$), $O=\Spec (k)\in T_1$, $R=k[[u,t]]$, $\cm =(u,t)\subset R$. 
\begin{defin}
\label{geomdataalt}
A geometric data is a triple $(X,j,\cf )$, where $X$ is an integral projective surface, 
$$j:T\rightarrow X$$
is a dominant $k$-morphism and $\cf\subset j_*\co_T$ is a quasicoherent subsheaf subject to the following conditions:
\begin{enumerate}
\item\label{GD1}
$j_*(T_1)=C\subset X$ is a curve\footnote{Notation: for a morphism of noetherian schemes $f:X\rightarrow Y$ and a closed subscheme $Z\subset X$, $f_*Z\subset Y$ is the closed subscheme defined by the ideal $\ker (\co_Y \stackrel{f^*}{\rightarrow}f_*\co_X\rightarrow f_*\co_Z)$}  (automatically integral),  
and $P=j(O)$ is a point neither in the singular locus of $C$ nor of $X$.
\item\label{GD2}
$T_1\times_X\{P\}=\{O\}$, $T\times_XC=rT_1$ (the fiber product is a subscheme of $T$ and $rT_1$ is an effective Cartier divisor on $T$), $r$ is called the rank of $(X,j,\cf )$.
\item\label{GD3}
There exists an effective, very ample Cartier divisor $C'\subset X$ with cycle $Z(C')=dC$ and for all $n>0$ the induced map (from the embedding $\cf \subset j_*\co_T$)
$$
H^0(X,\cf (nC'))\rightarrow H^0(X,j_*\co_T(nC'))=H^0(T,\co_T(ndrT_1))=Rt^{-ndr}\rightarrow Rt^{-ndr}/\cm^{ndr+1}t^{-ndr}
$$
is an isomorphism.
\end{enumerate}
\end{defin}

We left to the reader the proof of equivalence of these two definitions.

\begin{nt} 
\label{Kur}
With the data above we have the following properties:

1) $C$ is a $\dq$-Cartier divisor and $C^2=(C'\cdot C)/d=(C')^2/d^2$. 

2) $H^0(X,\cf )\simeq k$ (by \eqref{GD3}) hence we have a canonical embedding $\co_X\subset \cf$.

3) $\cf$ is a torsion free sheaf on $X$, and if $\cf$ is coherent then 
$$
\rk (\cf )(C^2)=r^2.
$$
Indeed, for $\cf$ as above we have 
$$
\chi (\cf (nC'))=\frac{(ndr+1)(ndr+2)}{2}.
$$
If $\cf$ is coherent of rank $m$ then 
$\cf\sim \co_X^m$ ($\sim$ means that the highest terms of the Hilbert polynomials of sheaves coincide). For any coherent sheaf $\cg$ on $X$ the function $\chi (\cg (nC'))$ is a polynomial of degree $\dim (\cg )=l$ with positive leading coefficient ($\in \dz /l!$), so $\chi (\cf (nC'))\sim m\chi (\co_X(nC'))$ and $n^2d^2r^2/2=m(C')^2/2$.  
\end{nt}

\begin{prop}
\label{Kurke}
If the embedding $\co_{X,P}\rightarrow \cf_P$ is an isomorphism, then $r=1$. Furthemore, $\co_X=\cf$ if and only if $X=\dpp^2$ and $C$ is a straight line in $\dpp^2$. 
\end{prop}

\begin{proof}
If $C$ is defined by $f=0$ in a small neighbourhood of $P$ (by \eqref{GD1} $\co_{X,P}$ is regular, and also $\co_{C,P}=\co_{X,P}/f\co_{X,P}$), then $fR=t^rR$ (by \eqref{GD2}) and $\cf (nC')_P=(\cf_P)_{f^{nd}}$. By \eqref{GD3} we have
$$
Rt^{-ndr}=H^0(X,\cf (nC'))\oplus \cm^{ndr+1}t^{-ndr},
$$
so $Rt^{-ndr}=\cf_Pt^{-nd} + \cm^{ndr+1}t^{-ndr}$ and if $\cf_P=\co_{X,P}$ we get $R=k[[u,t^r]]+\cm^{ndr+1}$ (for the proof we may assume that $u,t^r$ are generators of $\hat{\cm}_{X,P}=\cm_{X,P}\hat{\co}_{X,P}$). This is only possible for $r=1$. If $\co_X=\cf$ we get a canonical basis for each $H^0(X,\co_X(nC'))$ of the form $v_{ij}$, $0\le i\le i+j\le nd$, and $v_{ij}v_{hm}=v_{i+h,j+m}$ in $H^0(X\backslash C,\co_X)=:A$ ($v_{ij}$ corresponds to $u^it^j$ under the isomorphism in \eqref{GD3}). Thus $A=k[x,y]$ with $x=v_{10}$, $y=v_{01}$ (then $v_{ij}=x^iy^j$). 

Since 
$$X=\Proj \oplus_{n\ge 0}H^0(X,\co_X(nC'))=\Proj (\oplus_{n\ge 0}A_ns^n),$$ 
where $A_n=\sum_{i+j\le nd}kx^iy^j$, we get (by substituting $x=x'/z$, $y=y'/z$,
$s=z^d$) 
$$
\oplus_{n\ge 0}A_ns^n=k[(x')^i(y')^jz^k| i+j+k=m],
$$
i.e. $X=\dpp^2$ with the $d$-th Veronese embedding. Since $C^2=1$, we get $C$ is a straight line. 
\end{proof}

\subsection{Associated Schur pairs}
\label{pairs}

Given a geometric data $(X,C,P,\cf ,\pi , \phi )$ of rank $r$ we define a pair of subspaces
$$W,A\subset k[[u]]((t)),$$
where $A$ is a filtered subalgebra of $k[[u]]((t))$ and $W$ a filtered module over it, 
as follows (cf. \cite[Def.3.15]{Zhe2}): 

Let $f^d$ be a local generator of the ideal $\co_X(-C')_P$, where $C'=dC$ is a very ample Cartier divisor (cf. definition  \ref{geomdata}, item \ref{dat6}). Then $\nu (\pi (f^d))=(0,r^d)$ in the ring $k[[u,t]]$ and therefore  $\pi (f^d)^{-1}\in k[[u]]((t))$. So, we have natural embeddings for any $n >0$
$$
H^0(X, \cf (nC'))\hookrightarrow {\cf (nC')}_P\simeq f^{-nd} ({\cf }_P) \hookrightarrow k[[u]]((t)) \mbox{,}
$$
where the last embedding is the embedding $f^{-nd}{\cf }_P \stackrel{\phi }{\hookrightarrow } f^{-nd} k[[u,t]] {\hookrightarrow} k[[u]]((t))$ (cf. definition \ref{geomdata}, item~\ref{dat6}). Hence we have the embedding
$$
\chi_1 \; : \; H^0(X\backslash C, \cf )\simeq \limind_{n >0} H^0(X, \cf (nC')) \hookrightarrow k[[u]]((t)) \mbox{.}
$$
We define $W \eqdef \chi_1(H^0(X\backslash C, \cf ))$. Analogously the embedding $H^0(X\backslash C, \co )\hookrightarrow k[[u]]((t))$ is defined (and we'll denote it also by $\chi_1$). We define $A \eqdef \chi_1(H^0(X\backslash C, \co ))$.

As it follows from this construction, 
\begin{equation}
\label{qwerty}
A\subset k[[u']]((t')) \subset k[[u]]((t)), 
\end{equation}
where $t'=\pi (f)$, $u'=\pi (g)$ (see also definition \ref{geomdata}, item \ref{dat4}). Thus, on $A$ there is a filtration $A_n$ induced by the filtration ${t'}^{-n}k[[u']][[t']]$ on the space $k[[u']]((t'))$:
\begin{equation}
\label{qwerty1}
A_n = {A \cap {t'}^{-n}k[[u']][[t']]}= {A \cap {t}^{-nr}k[[u]][[t]]}
\end{equation}
We have $X\simeq \Proj (\tilde{A})$, where $\tilde{A}=\bigoplus\limits_{n=0}^{\infty}A_n s^n$ (see also \cite[lemma 3.3, lemma3.6, th.3.3]{Zhe2}). The similar filtration is defined on the space $W\subset k[[u]]((t))$:
\begin{equation}
\label{peresechenie}
W_n = {W \cap {t}^{-nr}k[[u]][[t]]}
\end{equation}
 And the sheaf $\cf \simeq \Proj (\tilde{W})$\footnote{Here and later in the article we use the non-standard notation $\Proj $ for the quasi-coherent sheaf associated with a graded module. If $M$ is a  filtered module, then we use the notation $\tilde{M}=\bigoplus\limits_{i=0}^{\infty} M_is^i$ for the analog of the Rees module, as well as for  filtered rings.}, where $\tilde{W}=\bigoplus\limits_{n=0}^{\infty}W_n s^n$. Note that we have $W_{nd}\simeq H^0(X, \cf (nC'))$ by definition \ref{geomdata}, item 6 and by construction of the map $\chi_1$. 

\subsubsection{Associated Schur pairs for alternative geometric data}

The same pair of subspaces $(W,A)$ can be defined also in terms of the alternative definition. Namely, to each geometric datum $(X,j,\cf )$ we associate a pair $(A,W)$ with $A\subset R[t^{-1}]=k[[u]]((t))$, $W\subset R[t^{-1}]$, where 
$A= H^0(X\backslash C, \co_X )\simeq \limind_{n >0} H^0(X, \co_X (nC'))$ embedded via 
$$
j^*:H^0(X, \co_X(nC'))\rightarrow H^0(X, j_*\co_T(nC'))=H^0(X,j_*\co_T(ndrT_1))=R\cdot t^{-ndr}
$$
and analogously for $\cf \subset j_*\co_T$. $A$ is a filtered subring with the filtration $A_n=A\cap R\cdot t^{-nr}$, and $W$ is a filtered $A$-module with the filtration $W_n=A\cap R\cdot t^{-nr}$. 

This pair $(A,W)$ determines the geometric data $(X,j,\cf )$, where $X$ and $\cf$ are defined as above, and the morphisms $j:T\rightarrow X$, $\cf \subset  j_*\co_T$ come from the embeddings $A\subset k[[u]]((t))$, $W\subset k[[u]]((t))$. 

\subsection{Category of geometric data}

In this section we recall definition of a category $\cq$ of geometric data from \cite{Zhe2} and give its alternative definition. Formally we need this section  only to recall the content of theorem \ref{dannye2}. We write it for the sake of completeness. 

\begin{defin}
\label{geomcategory}
We define a category $\cq$ of geometric data as follows:
\begin{enumerate}
\item\label{cat1}
The set of objects is defined by
$$
Ob (\cq )=\bigcup_{r\in \sdn} \cq_r,
$$
where $\cq_r$ denotes the set of geometric data of rank $r$. 
\item\label{cat2}
A morphism
$$
(\beta , \psi ) \, : \, [(X_1,C_1,P_1,\cf_1 ,\pi_1 , \phi_1 )] \longrightarrow [(X_2,C_2,P_2,\cf_2 ,\pi_2 , \phi_2 )]
$$
of two objects
consists of a morphism $\beta :X_1\rightarrow X_2$ of surfaces and a homomorphism $\psi :\cf_2\rightarrow \beta_*\cf_1$ of sheaves on $X_2$ such that:
\begin{enumerate}
\item\label{cat1.1}
$\beta |_{C_1} :C_1\rightarrow C_2$ is a morphism of curves and $\beta^{-1}(X_2\backslash C_2)=X_1\backslash C_1$;
\item\label{cat1.2}
$$
\beta (P_1)=P_2.
$$
\item\label{cat1.3}
There exists a continuous $k$-algebra isomorphism $h:k[[u,t]] \rightarrow k[[u,t]]$ (in a natural linear topology, where the base of neighbourhoods of zero is generated by the powers of the maximal ideal) such that
$$
h(u)=u \mbox{\quad mod \quad} (u^2)+(t), \mbox{\quad} h(t)=t \mbox{\quad mod \quad} (ut)+(t^2),
$$
and the following commutative diagram holds:
$$
\begin{CD}
H^0(X_2\backslash C_2, \co_2) @>\beta^{\sharp}>> H^0(X_1\backslash C_1, \co_1) \\
@VV\chi_2V @VV\chi_1V \\
k[[u]]((t)) @>\hat{h}>> k[[u]]((t)), 
\end{CD}
$$
where $\hat{h}$ denotes the natural extension of the map $h$ to a $k$-algebra $k[[u]]((t))$ automorphism.
\item\label{cat1.4}
There is a $k[[u,t]]$-module isomorphism $\xi :k[[u,t]] \simeq h_*(k[[u,t]])$ (which is given just by  multiplication of a single invertible element $\xi\in k[[u,t]]^*$) such that the following commutative diagram holds:
$$
\begin{CD}
H^0(X_2\backslash C_2, \cf_2) @>\psi>>H^0(X_2\backslash C_2, {\beta}_*\cf_1) = H^0(X_1\backslash C_1, \cf_1)\\
@VV\chi_2V @VV\chi_1V \\
k[[u]]((t)) @>\hat{\xi} >> h_*(k[[u]]((t)))=k[[u]]((t)).
\end{CD}
$$
\end{enumerate}
\end{enumerate}
\end{defin}

\subsubsection{Alternative definition of the category}

We can give an alternative definition of the category as follows. The set of objects is defined as before, i.e. an object from $\cq_r$ denotes the geometric datum $(X,j,\cf )$ of rank $r$.

We define a morphism of two objects $(X_1,j_1,\cf_1 )\rightarrow (X_2,j_2,\cf_2)$ as a pair $(\beta ,\psi )$ with $\beta :X_1\rightarrow X_2$ a dominant morphism of surfaces,  $\psi :\cf_2\rightarrow \beta_*\cf_1$ a morphism of quasicoherent sheaves subject to the following conditions:
\begin{enumerate}
\item 
$(\beta^{-1}C_2)_{red}=C_1$
\item
there exists $h\in Aut_k(T)$ with $h_*(T_1)=T_1$,
$$
h^*(u)=u \mbox{\quad mod \quad} (u^2)+(t), \mbox{\quad} h^*(t)=t \mbox{\quad mod \quad} (ut)+(t^2),
$$
such that the diagram 
$$
\begin{CD}
T @>j_1>>X_1\\
@VV h V @VV\beta V \\
T @>j_2 >> X_2
\end{CD}
$$
is commutative;
\item 
there exists $\xi\in Aut_{\co_T}(\co_T)$ (i.e. an element from $R^*$) such that the diagram 
$$
\begin{CD}
(\cf_2)_{P_2} @>\psi >>(\cf_1)_{P_1} \\
@VV \cap V @VV\cap V \\
R @>\xi >> R
\end{CD}
$$
is commutative.
\end{enumerate} 

The composition with a second morphism $(\beta' ,  \psi' )$ is given by 
$(\beta', \psi')\circ (\beta ,\psi )=(\beta' \beta ,(\beta')_*(\psi )\psi')$.

\begin{nt}
\label{Kur1}
It is interesting to note that, in general, a morphism of pairs $\beta :(X_1,C_1)\rightarrow (X_2,C_2)$ induces an embedding $H^0(X_2\backslash C_2,\co_{X_2})\hookrightarrow H^0(X_1\backslash C_1,\co_{X_1})$ if and only if $(\beta^{-1}C_2)_{red}=C_1$. 

The "if" part is obvious, let's show the "only if" part. Without loss of generality let $d>0$ be an integer such that $C_1'=dC_1$ and $C_2'=dC_2$ are effective very ample Cartier divisors. Then $\beta^*C_1'=mC_2'+E$, where $E=0$ or $E$ is an effective Cartier divisor. 

If $A=H^0(X_1\backslash C_1,\co_{X_1})$, $q=H^0(X_1\backslash C_1,\co_{X_1}(-E))$ (so, $q$ is an invertible ideal), then 
$\Spec (A)=X_1\backslash C_1$, $\Spec (\cap_n q^{-n})=X_1\backslash \beta^{-1}(C_2)$. If $B=H^0(X_2\backslash C_2,\co_{X_2})$, then we have $A\subset \cap_nq^{-n}$, $\cap_nq^{-n}$ is finite over $B$. If $B\subset A$, then from the exact triple 
$$
0\rightarrow A/B \rightarrow (\cap_nq^{-n})/B \rightarrow (\cap_nq^{-n})/A \rightarrow 0
$$ 
it follows that $(\cap_nq^{-n})$ must be a finite $A$-module, a contradiction if $E\neq 0$. 
\end{nt}

\begin{nt}
\label{nt}
The condition in item 2 on $h^*(u)$, $h^*(t)$ is important to establish a categorical equivalence with the category of Schur pairs from \cite{Zhe2}. The reason is that 
automorphisms of the form $h^*(u)=c_1u$, $h^*(t)=c_2t$ applied to a Schur pair will lead  (after application of the quasi-inverse functor from  \cite[Th.3.3]{Zhe2}), to a datum with another sheaf $\cf$ (i.e. to a datum not isomorphic to the original one). This effect was known already in the classical KP theory as a scaling transform (cf. \cite[\S 4,\S 7]{SW}). 
\end{nt}

\subsection{The restriction map $\zeta$}
\label{mapzeta}

To construct the map $\zeta$ mentioned in Introduction we need to extend the constructions from  previous section to a wider set of sheaves. Let $X,C,C',P$ and $\co_{X,P}\subset R$ (for the embedding $\pi$ or for the morphism $j: T\rightarrow X$) be as before. Let also $A\subset k[[u]]((t))=R[t^{-1}]$ be as before. 
We start with the following remark. 

\begin{nt}
\label{key} 
Let's note that we can construct the analogous spaces $W_n,\tilde{W}$ for any torsion free sheaf $\cf$ (not only for sheaves from data) endowed with a $\co_P$-module embedding $\cf_P\hookrightarrow k[[u,t]]$. The most important example of such sheaf with an embedding is a coherent torsion free rank one Cohen-Macaulay sheaf $\cf$ on $X$ (or, more generally, $\cf$ is locally free of rank one at $P$), where we additionally assume that the rank of data $r=1$ (i.e. the embedding $\pi$ gives an isomorphism $\co_{X,P}\simeq R$). 
 In this case the stalk $\cf_P$ is a free $\co_P$-module. Let  $\phi': \cf_P\simeq \co_P$ be a trivialisation;  we can define the embedding $\phi$ by composing a trivialization $\phi'$ with the isomorphism $\pi$. 
Note that, if we choose another trivialisation of such sheaf $\cf$, then the new space $W$ will differ from the old one  by multiplication on an element $a\in k[[u,t]]^*$ and the space $A$ will not change. 
Note also that the property $W_{nd}\simeq H^0(X, \cf (nC'))$ might not be true in general. 
\end{nt}

Further we will use  the  following notation. If $W\subset R[t^{-1}]$ is an $A$-module, we get a filtration $W_n=t^{-nr}R\cap W$ (compatible with the filtration on $A$) and a fortiori graded $\tilde{A}$-modules 
$$
\tilde{A}(i) \mbox{\quad} (\tilde{A}(i)_k=\tilde{A}_{k+i}), \mbox{\quad} \tilde{W}(i) \mbox{\quad} (\tilde{W}(i)_k=\tilde{W}_{k+i})
$$
and quasicoherent sheaves on $X$:
$$
\cb_i=\Proj (\tilde{A}(i)), \mbox{\quad} \cf_i=\Proj (\tilde{W}(i))
$$
with $\cb_i\subset \cb_{i+1}$, $\cf_i\subset \cf_{i+1}$. Note that
\begin{enumerate}
\item
$\cb_{id}\simeq \co_X(iC')$, and if $W$ comes from a geometric datum, then $\cf_{id}\simeq \cf (iC')$. In general, $\cf_{nd}\simeq \cf_0(dC')$, because by \cite[prop.2.4.7]{EGAII} we have
$\cf_{nd}=\Proj (\tilde{W}(nd))\simeq \Proj (\tilde{W}^{(d)}(n))$ and $\Proj (\tilde{W}^{(d)}(n))\simeq \Proj (\tilde{W}^{(d)})(n)\simeq \cf_0 (nC')$ for any $n$.
\item
If $\cf$ is a quasicoherent sheaf with an embedding $\cf \subset j_*\co_T$ (equivalently $\cf_P\subset R$) inducing $W=H^0(X\backslash C,\cf )\subset R[t^{-1}]$, then $\cf (iC')\subseteq \cf_{id}$. 
\item
If $\cf$ is a torsion free quasicoherent sheaf and if $\cf_P$ is a free module of rank one, we can find an embedding $\cf_P\subset R$ (by a choice of a generator $\cf_P\simeq \co_{X,P}\subset R$). The resulting sheaves $\cf_i$ do not depend on the choice of the generator, up to isomorphisms compatible with the embeddings $\cf\subset \cf_i\subset \cf_{i+1}$.
\item
From \eqref{qwerty}, \eqref{qwerty1} it easily follows that the sheaves 
$$
\cb_i/\cb_{i-1}\simeq \Proj (\bigoplus_{n=0}^{\infty} A_{i+n}/A_{i+n-1}),\mbox{\quad} \cf_i/\cf_{i-1}\simeq \Proj (\bigoplus_{n=0}^{\infty} W_{i+n}/W_{i+n-1})
$$ 
are torsion free coherent sheaves on
$C\simeq \Proj (\cb_0/\cb_{-1})$\footnote{We mean here and below the pull-backs of the factor-sheaves on $C$. Note that that these pull-backs are canonically isomorphic to the sheaves $\Proj (\bigoplus_{n=0}^{\infty} A_{i+n}/A_{i+n-1})$, $\Proj (\bigoplus_{n=0}^{\infty} W_{i+n}/W_{i+n-1})$, where the graded modules are considered as $\cb_0/\cb_{-1}$-modules. Indeed, for any $f\in A_d$ and any graded $\tilde{A}$-module $M$ we have $(M\otimes_{\tilde{A}}(\cb_0/\cb_{-1}))_{(f)}\simeq M_{(f)}\otimes_{\tilde{A}_{(f)}}(\cb_0/\cb_{-1})_{(f)}$, where from the pull-backs of the sheaves $\cb_i/\cb_{i-1}$, $\cf_i/\cf_{i-1}$ are isomorphic to the sheaves $\Proj (M)$, $\Proj (N)$ on $C$, where $M=    (\bigoplus_{n=0}^{\infty} A_{i+n}/A_{i+n-1})\otimes_{\tilde{A}}(\cb_0/\cb_{-1})$, $N=(\bigoplus_{n=0}^{\infty} W_{i+n}/W_{i+n-1})\otimes_{\tilde{A}}(\cb_0/\cb_{-1})$. But the modules $M$, $N$ are isomorphic to the $\cb_0/\cb_{-1}$-modules $(\bigoplus_{n=0}^{\infty} A_{i+n}/A_{i+n-1})$,  $(\bigoplus_{n=0}^{\infty} W_{i+n}/W_{i+n-1})$. }.
\end{enumerate}

\begin{defin}
For any torsion free sheaf $\cf$ endowed with a $\co_{P}$-module embedding $\cf_P\hookrightarrow k[[u,t]]$  we define a map (a kind of a "restriction" map): 
\begin{equation}
\label{zeta}
\zeta :\cf \mapsto \cf_0/\cf_{-1}
\end{equation}
from the set of torsion free sheaves on $X$ to the set of torsion free sheaves on $C$. 
\end{defin}

\begin{nt}
\label{torsion_free_sheaves}
For sheaves $\cf$ satisfying the property $W_{nd}\simeq H^0(X, \cf (nC'))$ for all $n\gg 0$ we have $\cf \simeq \cf_0$ by \cite[Lemma 9]{Pa} and \cite[Ch.2, ex. 5.9]{Ha}. If $\cf$ is a torsion free sheaf of rank one locally free at $P$, then by remark \ref{key} for another choice of trivialisation at $P$ there are isomorphisms $\cf_k'\simeq \cf_k$ for any $k$. So, in this case definition of $\zeta$ don't depend on trivialisation. In fact, in this case it depends only on $\cf$ (see remark \ref{rem2.8} below). 

On the other hand, for any torsion free sheaf $\cf$ and any $m>0$ we also have the exact sequences
$$
0\rightarrow \cf\otimes_{\co_X}\co_X(-mC')\rightarrow \cf \rightarrow \cf \otimes_{\co_X}(\co_X/\co_X(-mC'))\rightarrow 0.
$$
Thus the pull-back of the sheaf $\cf_0$ on the scheme $(C,i^{-1}(\co_X/\co_X(-mC')))$ (where $i:C\hookrightarrow X$ denotes the embedding)  is isomorphic to the pull-back of the sheaf $\cf_{0}/\cf_{-md}$. 

{\it Further we will denote the pull-back of a sheaf $\cf$ on the scheme $(C,i^{-1}(\co_X/\co_X(-mC')))$ as} $\cf |_{mC'}$. 

Note that the scheme $(C,i^{-1}(\co_X/\co_X(-mC')))$ is an irreducible scheme since $C$ is irreducible. Hence the nilradical of the ring $\tilde{A}/\tilde{A}(-md)$ is prime. 
From \eqref{qwerty}, \eqref{qwerty1} it again follows that $Ass(\tilde{W}/\tilde{W}(-md))$ coincides with this nilradical. Therefore, the restriction $\cf_0 |_{mC'}$ is pure of dimension one (cf. \cite[p.3]{HL}), because any restriction of a non-zero section $a\in \cf_0 |_{mC'}(U)$ (where $U$ is any open subset of $C$) to a smaller open subset is not zero. Thus, for any torsion free sheaf of rank one locally free at $P$ the sheaf $\cf |_{mC'}\subset \cf_0|_{mC'}$ is pure of dimension one. 
\end{nt}

Notably, we have the following property for any torsion free sheaf $\cf$ such that its restriction $\cf |_{mC'}$ on the scheme $(C,i^{-1}(\co_X/\co_X(-mC')))$ is pure of dimension one:  

\begin{lemma}
\label{claim1}
Let $\cf$ be a torsion free sheaf on $X$ endowed with a $\co_{P}$-module embedding $\cf_P\hookrightarrow k[[u,t]]$ satisfying the following condition: if $w\in W_{nd}$, then $w\in f^{-nd}(\cf_P)$. Assume that its restriction $\cf |_{mC'}$ is pure of dimension one for all $m>0$. 

Then we have $H^0(X, \cf (nC'))\simeq W_{nd}$ for all $n\ge 0$.
\end{lemma}

\begin{nt}
\label{rem2.7}
The condition on a $\co_{P}$-module embedding from lemma is satisfied for example for all 
rank one torsion free sheaves locally free at $P$ (see remark \ref{key}) and for coherent sheaves of rank $r$ from the geometric data (where $r$ coincides with the rank of the data), because $\co_P$ is a regular factorial ring. Other examples see in theorem \ref{CMmodules}. 
\end{nt}

\begin{proof}
By definition of the space $W$ we have 
$$
W_{nd}=\{w\in W | f^{nd}w\in k[[u]][[t]]\}=\{w\in W | \nu_t(f^{nd}w)\ge 0\}.
$$
We also have by definition $\chi_1(H^0(X, \cf (nC'))\subset W_{nd}$. Let $w\in W_{nd}$, $w\neq 0$. Let's show that $w\in \chi_1(H^0(X, \cf (nC'))$.  
We have 
$$
w\in \chi_1(H^0(X, \cf (mC'))
$$ 
for some $m$. Since $\cf $ is a torsion free sheaf and $C'$ is a Cartier divisor, we have embeddings 
$$
\chi_1(H^0(X, \cf (kC'))\subset \chi_1(H^0(X, \cf (nC'))
$$
for all $k\le n$. Suppose that $m>n$. 
 Assume the converse: $w\notin \chi_1(H^0(X, \cf (nC'))$. Let $b\in H^0(X, \cf (mC'))$ be the preimage of $w$: $w=\chi_1(b)$. 
 
There is a neighbourhood $U(P)$ of the point $P$, where the ample Cartier divisor $C'$ is defined by the element $f^d$. Since $w\in W_{nd}$, we have $w\in f^{-nd}(\cf_{P})$, thus $b|_{U(P)}\in \Gamma (U(P), \cf (nC'))$ and $b|_{U(P)}\neq 0$ (since $\cf $ is torsion free). Now we have the following commutative diagram:
$$
\begin{array}{cccc}
&b& \hookrightarrow & H^0(C, \cf (mC')|_{(m-n)C'})\\
&\downarrow & & \downarrow \\
0\rightarrow \Gamma (U(P), \cf  (nC')) \rightarrow & \Gamma (U(P),\cf (mC')) & \stackrel{\alpha}{\rightarrow} & H^0(U(P)\cap C, \cf (mC')|_{(m-n)C'})
\end{array},
$$
where the vertical arrows are embeddings. Indeed, the right vertical arrow is an embedding since $\cf (mC')|_{(m-n)C'}$  is pure of dimension one by assumption. 

But $\alpha (b)=0$, a contradiction. Thus, $b\in H^0(X, \cf (nC'))$.
\end{proof}

By \cite[Cor. 3.1]{Ku3} all sheaves $\cf$ of rank one appearing in the geometric data from definition \ref{geomdata} are Cohen-Macaulay along $C$. As it easily follows from definition \ref{geomdata} (item 6) all such sheaves fulfil the property $\co_X\subset \cf$, $P\notin\Sup (\cf /\co_X)$. 
\begin{lemma}
\label{claim2}
Let $\cf$ be a torsion free rank one sheaf on $X$. Assume that $\cf$ is Cohen-Macaulay along $C$.

Then for some trivialisation $\hat{\phi} :\hat{\cf}_P\simeq k[[u,t]]$ (see remark \ref{key})
$$
W_{nd}\simeq H^0(X,\cf (nC'))
$$
for all $n\ge 0$, or, equivalently, $\cf_0\simeq \cf$. 
\end{lemma}

\begin{proof}
The proof follows immediately from remarks \ref{torsion_free_sheaves}, \ref{rem2.7} and lemma \ref{claim1}, since $\cf$ is locally free at $P$. 
\end{proof}

\begin{nt}
\label{rem2.8}
If $\cf$ is a torsion free sheaf of rank one locally free at $P$, then $\zeta (\cf )\simeq i^*(\cf )$, where $i$ is the same as in remark \ref{torsion_free_sheaves}. Indeed, $\cf\simeq \cf_0$ by lemma \ref{claim1}, remarks \ref{torsion_free_sheaves}, \ref{rem2.7}. By the arguments from footnote 3 $i^*(\cf_0)\simeq \Proj (\tilde{W}\otimes_{\tilde{A}}(\tilde{A}/\tilde{A}(-1)))$. It is easy to see that the $(\tilde{A}/\tilde{A}(-1))$-modules $(\tilde{W}\otimes_{\tilde{A}}(\tilde{A}/\tilde{A}(-1)))$ and $(\tilde{W}/\tilde{W}(-1)\otimes_{\tilde{A}}(\tilde{A}/\tilde{A}(-1)))$ are isomorphic. But again by the arguments from footnote 3 $i^*(\cf_0/\cf_{-1})\simeq \Proj (\tilde{W}/\tilde{W}(-1)\otimes_{\tilde{A}}(\tilde{A}/\tilde{A}(-1)))$. 
\end{nt}

\begin{corol}
\label{cor1}
For any $k\ge 0$ we have $H^0(X, \cf_k(nC'))\simeq W_{nd+k}$ for all $n\ge 0$. 
\end{corol}

The proof is obvious.

\subsection{Commutative rings of operators}
\label{CRO}

In this paper we will work mainly with commutative $k$-algebras of PDOs $B\subset D=k[[x_1,x_2]][\partial_1,\partial_2]$ that satisfy the following condition:
\begin{multline}
\label{property}
\mbox{$B$ contains the operators $P,Q$ with constant principal symbols such that}\\
\mbox{ the intersection of the characteristic divisors of $P,Q$ is empty. }
\end{multline} 

Recall that the symbol $\sigma (P)$ of an operator $P\in D$ is called constant if $\sigma (P)\in k[\xi_1,\xi_2]$. The characteristic divisor is given by the divisor of zeros of $\sigma (P)$ in $\dpp_k^{n-1}$. It is unchanged by a $k$-linear change of coordinates $x_1,\ldots ,x_n$. Recall also that any operator $Q$ from the ring $B$ satisfying condition \eqref{property} has constant principal symbol (see e.g. \cite[Lemma 2.1]{Ku3}) and all such rings are finitely generated $k$-algebras of Krull dimension 2 (see e.g. \cite[Th.2.1]{Ku3}). 

In the work \cite{Zhe2} was shown that such algebras are a part of a wider set of commutative $k$-algebras $B'\subset \hat{D}$, and all algebras from this set can be classified in terms of geometric data from subsection \ref{data}. To explain what is going on we need to recall several definitions and statements from  \cite{Zhe2}. 

\begin{defin}
\label{defin3.0}
We define the order function on the ring $k[[x_1,x_2]]$ by the rule 
$$
\ord_M(a)=\sup \{n|a\in (x_1,x_2)^n\}.
$$
\end{defin}

\begin{defin}{(\cite[Sect.2.1.5]{Zhe2})}
Define
\begin{multline}
\hat{D}_1=\{a=\sum_{q\ge 0} a_{q}\partial_1^q\mbox{\quad }|  a_q\in k[[x_1, x_2]] \mbox{ and for any $N\in \dn$ there exists $n\in \dn$ such that }\\
\mbox{$\ord_{M}(a_m)>N$ for any $m\ge n$}\}. 
\end{multline}

Define
$$
\hat{D}=\hat{D}_1[\partial_2], \mbox{\quad} \hat{E}_+=\hat{D}_1((\partial_2^{-1})).
$$
\end{defin}

\begin{defin}{(\cite[Def.2.12]{Zhe2})}
\label{defin3}

We say  that an operator $P\in \hat{D}$ has order $\ord_{\Gamma} (P)=(k,l)$ if $P=\sum_{s=0}^lp_s\partial_2^s$, where $p_s\in \hat{D}_1$, $p_l\in k[[x_1,x_2]][\partial_1]=D_1$, and $\ord (p_l)=k$ (here $\ord$ is the usual order in the ring of differential operators $D_1$). In this situation we say that the operator $P$ is monic if the highest coefficient of $p_l$ is $1$. 

We say that an operator $Q=\sum q_{ij}\partial_1^{i}\partial_2^{j}\in \hat{E}_+$ {\it satisfies the condition $A_{1}(m)$} if $\ord_M(q_{ij})\ge i+j-m$ for all $(i,j)$. 

An operator $P\in \hat{D}$, $P=\sum p_{ij}\partial_1^{i}\partial_2^{j}$ with $\ord_{\Gamma} (P)=(k,l)$ {\it satisfies the condition $A_{1}$} if it satisfies $A_1(k+l)$. 
\end{defin}

\begin{defin}{(\cite[Def.2.18]{Zhe2})}
\label{elliptic}
The ring $B\subset \hat{D}$ of commuting operators is called quasi elliptic if it contains two monic operators $P,Q$ such that $\ord_{\Gamma} (P)= (0,k)$  and $\ord_{\Gamma} (Q)=(1,l)$ for some $k,l\in \dz$.  

The ring $B$ is called $1$-quasi elliptic if $P,Q$ satisfy the condition $A_{1}$. 
\end{defin}

\begin{defin}{(\cite[Def.3.4]{Zhe2})}
\label{rings}
The commutative $1$-quasi elliptic rings $B_1$, $B_2\subset \hat{D}$ are said to be equivalent if there is an invertible operator $S\in \hat{D}_1$ of the form $S=f+S^-$, where $S^-\in \hat{D}_1\partial_1$, $f\in k[[x_1,x_2]]^*$,  such that $B_1=SB_2S^{-1}$. 
\end{defin}

\begin{defin}{(\cite[Def.3.1]{Zhe2})}
\label{1space}
The subspace $W\subset  k[z_1^{-1}]((z_2))$ is called $1$-space, if there exists a basis $w_i$ in $W$ such that $w_i$ satisfy the condition $A_{1}$ for all $i$ (we identify here and below the ring $k[z_1^{-1}]((z_2))$ with the ring $k[\partial_1]((\partial_2^{-1}))$ via $z_1 \leftrightarrow \partial_1^{-1}$, $z_2\leftrightarrow \partial_2^{-1}$)
\end{defin}

Using the identification $z_1 \leftrightarrow \partial_1^{-1}$, $z_2\leftrightarrow \partial_2^{-1}$ we can extend the definition of the order function $\ord_{\Gamma}$ from definition \ref{defin3} on the field $V=k((z_1))((z_2))$. Using the anti-lexicographical order on the group $\dz\oplus\dz$ we define the lowest term $LT(a)$ of any series $a$ from $V$ to be the monomial of $a$ with the lowest order. 

\begin{defin}{(\cite{ZO})}
\label{defin2}
The support of a $k$-subspace $W$ from the space $V$ is the $k$-subspace $\Sup (W)$ in the space $V$ generated by $\LT (a)$ for all $a\in W$. 
\end{defin}

\begin{defin}{(\cite[Def.3.2]{Zhe2})}
\label{sch}
We say that a pair of subspaces $(A,W)$, where $A,W \subset  k[z_1^{-1}]((z_2))$ and $A$ is a $k$-algebra with unity such that $W\cdot A\subset W$, is a $1$-Schur pair if $A$ and $W$ are $1$-spaces and $\Sup (W)=k[z_1^{-1},z_2^{-1}]$. 

We say that $1$-Schur pair is a $1$-quasi elliptic Schur pair if $A$ is a $1$-quasi elliptic ring. 
\end{defin}

Consider the ring $\hat{E}_+=\hat{D}_1((\partial_2^{-1}))$. It has a natural action on the space $k[z_1^{-1}]((z_2))$ via the isomorphism $\hat{E}_+/(x_1,x_2)\hat{E}_+\simeq k[z_1^{-1}]((z_2))$ which endows this space with the structure of a right $\hat{E}_+$-module. 

\begin{defin}{(\cite[Def.3.3]{Zhe2})}
\label{admissible}
An operator $T\in \hat{E}_+$ is said to be admissible if it is an invertible operator of order zero such that $T\partial_1 T^{-1}$, $T\partial_2T^{-1}\in k[\partial_1]((\partial_2^{-1}))$. The set of all admissible operators is denoted by $\Adm$. 

An operator $T\in \hat{E}_+$ is said to be $1$-admissible if it is admissible and 
satisfies the condition $A_{1}$ (the definition of the condition $A_{1}$ for operators from $\hat{E}_+$ is literally the same as for operators from $\hat{D}$). The set of all $1$-admissible operators is denoted by $\Adm_{1}$.

We say that two $1$-Schur pairs $(A,W)$ and $(A',W')$ are equivalent if $A'=T^{-1}AT$ and $W'=WT$, where $T$ is an admissible operator. 
\end{defin}

\begin{nt}
\label{alpha=1}
In \cite{Zhe2} a more general growth condition $A_{\alpha}$, $\alpha \ge 1$, was introduced. In the present paper we use only the condition $A_1$.  
This is the only case when the classification theorems from \cite{Zhe2} (see also below) work. 

Let's recall here one more notion from \cite{Zhe2}.

Consider the set in $\hat{E}_+$
$$
\Pi =\{P\in \hat{E}_+| \mbox{\quad $\exists $  $m\in \dz_+$ s. that $P$ satisfies $A_{1}(m)$}\} .
$$
It is an associative subring with unity (see \cite[Corol.2.2]{Zhe2}).

We note that $\Pi \supset D$. Recall that by \cite[lemma 2.10, lemma 2.11]{Zhe2} it follows that any $1$-quasi elliptic ring $B$ belongs to $\Pi$.  
\end{nt}

\begin{nt}
\label{admnote}
By \cite[Lemma 2.11]{Zhe2} any two operators  with constant coefficients $L_1,L_2$ of the form
$$
L_1=\partial_1 +\sum_{q=1}^{\infty}v_q\partial_2^{-q}, \mbox{\quad } L_2=\partial_2+\sum_{q=1}^{\infty}u_q\partial_2^{-q} 
$$
and satisfying condition $A_1$ can be obtained as $L_1=S^{-1}\partial_1S$, $L_2=S^{-1}\partial_2S$, where $S=1+S^-\in k[[x_1,x_2]][\partial_1]((\partial_2^{-1}))$ is an invertible zeroth order $1$-admissible operator. 

On the other hand, as one can easily check, for the operator 
\begin{equation}
\label{special_form}
T_0 =c_0\exp (c_1x_2 \partial_1)\exp (c_2x_2+c_3x_1) \in \hat{D}_1,
\end{equation}
where $c_0,c_1,c_2,c_3\in k$, we have 
$$
T_0^{-1}\partial_1T_0=\partial_1+c_3, \mbox{\quad} T_0^{-1}\partial_2T_0=\partial_2+c_1\partial_1+c_1c_3+c_2.
$$
So, any $1$-admissible operator can be written in the form $T=ST_0$. 
\end{nt}

\begin{theo}{(\cite[Th.3.2]{Zhe2})}
\label{schurpair} 
There is a one to one correspondence between the classes of equivalent $1$-quasi elliptic Schur pairs $(A,W)$  with $\Sup(W)=\langle z_1^{-i}z_2^{-j}\mbox{\quad }| i,j\ge 0 \rangle$ and the classes of equivalent  $1$-quasi elliptic rings  of commuting operators $B\subset \hat{D}$.    
\end{theo}

The proof of the theorem is constructive; the spaces $A$ and $W$ are obtained as follows: $A=S^{-1}BS$, $W=k[z_1^{-1},z_2^{-1}]S$, where $S$ is a monic operator of special type satisfying the condition $A_1$. It is defined by a pair of normalized operators from $B$ (see \cite[\S 2.3.4]{Zhe2} or definition below) using the analogue of Schur's theorem in dimension one (see \cite[Lemma 2.11]{Zhe2}). 

\begin{defin}
\label{defin4,5}
We say that commuting monic operators $P,Q\in \hat{E}_+$ with $\ord_{\Gamma} (P)=(0,k)$, $\ord_{\Gamma} (Q)= (1,l)$  are almost normalized if 
$$P=\partial_2^k+ \sum_{s=-\infty}^{k-1}p_{s}\partial_2^{s} \mbox{\quad} Q=\partial_1\partial_2^l+ \sum_{s=-\infty}^{l-1}q_{s}\partial_2^{s},$$ where $p_s,q_s\in \hat{D}_1$.  

We say that $P,Q$ are  normalized if 
$$P=\partial_2^k+ \sum_{s=-\infty}^{k-2}p_{s}\partial_2^{s} \mbox{\quad} Q=\partial_1\partial_2^l+ \sum_{s=-\infty}^{l-1}q_{s}\partial_2^{s},$$ where $p_s,q_s\in \hat{D}_1$.  
\end{defin}

Recall that by \cite[lemma 2.10]{Zhe2} any two commuting operators of order $(0,k)$ and $(1,l)$ can be normalized by conjugating with an invertible operator $S\in \hat{D}_1$. 
The space $A$ from theorem \ref{schurpair} depends only on the choice of the pair of normalized operators from $B$, and don't depend on the choice of the operators $S$ from \cite[Lemma 2.11]{Zhe2}. If one chooses another pair of normalized operators from $B$, then the resulting Schur pair from theorem \ref{schurpair} will be equivalent to the first one. 
The following lemma clarifies the structure of elements in a ring that has a pair of normalized operators and in any equivalent ring. 
\begin{lemma}
\label{elementary_calculations}
i) If the ring $B\subset \Pi \cap \hat{D}$ of commuting operators contains a pair of normalized operators $P,Q$ with $\ord_{\Gamma}(P)=(0,k)$, $\ord_{\Gamma}(Q)=(1,l)$ ($k>0$), then all operators in $B$ have constant highest coefficients, i.e. if $L=\sum_{s=0}^Nl_s\partial_2^s$, then $l_N$ is an operator with constant coefficients. In particular, $l_N\in D_1$ (i.e. it has a finite order).

Moreover, 
any operator $P'\in B$ with $\ord_{\Gamma}(P')=(0,m)$ has the form
$$
P'=\sum_{s=0}^m p_s'\partial_2^s, \mbox{\quad where $p_m'\in k$ and $p_{m-1}'$ has constant coefficients} 
$$
and any operator $Q'\in B$ with $\ord_{\Gamma}(Q')=(1,n)$ has the form 
$$
Q'=\sum_{s=0}^n q_s'\partial_2^s, \mbox{\quad where $q_{n}'=c_1\partial_1+c_0$, $c_0,c_1\in k$}.
$$

ii) If $B'=S^{-1}BS$, $S\in \hat{D}_1$ is an equivalent $1$-quasi elliptic ring containing a pair of normalized operators $P',Q'$ with $\ord_{\Gamma}(P')=(0,k')$, $\ord_{\Gamma}(Q')=(1,l')$ ($k'>0$), then $S$ has the form 
$$
S =c_0\exp (c_1x_2 \partial_1)\exp (c_2x_2+c_3x_1) \in \hat{D}_1,
$$
where $c_0,c_1,c_2,c_3\in k$ (cf. remark \ref{admnote}). 
\end{lemma}

\begin{proof}
i) We have
\begin{equation}
\label{vych1}
0=[P,P']=k\partial_2(p_m')\partial_2^{k+m-1}+k\partial_2(p_{m-1}')\partial_2^{k+m-2}+[p_{k-2},p_m']\partial_2^{k+m-2}+
\mbox{\quad terms of lower degree}.
\end{equation}
Hence $\partial_2(p_m')=0$, i.e. $p_m'$ don't depend on $x_2$. Then we have
\begin{equation}
\label{vych2}
0=[Q,P']=[\partial_1,p_m']\partial_2^{m+l}+[\partial_1,p_{m-1}']\partial_2^{l+m-1}+[q_{l-1},p_m']\partial_2^{l+m-1}+
\mbox{\quad terms of lower degree}.
\end{equation}
Hence $[\partial_1,p_m']=0$ and therefore $p_m'$ must be an operator with constant coefficients. So, $p_m'\in D_1$ (and  clearly these arguments work for any operator from $B$). Since $\ord_{\Gamma}(P')=(0,m)$, $p_m'$ is a constant, and since $\ord_{\Gamma}(Q')=(1,n)$, $q_n'$ must be a linear polynomial.  
But then from \eqref{vych2} we have $[\partial_1,p_{m-1}']=0$, i.e. $p_{m-1}'$ don't depend on $x_1$, and from \eqref{vych1} we have $\partial_2(p_{m-1}')=0$, i.e.  $p_{m-1}'$  must be an operator with constant coefficients. 

ii) We have 
$P'=S^{-1}\tilde{P}S$, $Q'=S^{-1}\tilde{Q}S$ for some operators $\tilde{P}$, $\tilde{Q}\in B$. Since $S$ is invertible, we obviously have 
$$
S=c \in k^* \mbox{\quad} \mod (x_1,x_2). 
$$
Therefore, since by item i) the highest terms of operators $\tilde{P}, \tilde{Q}$ are constant-coefficient operators, we must have $\ord_{\Gamma}(\tilde{P})=(0,k')$  and 
$\ord_{\Gamma}(\tilde{Q})=(1,l')$. From remark \ref{admnote} we know that there exists an operator $S_0$ of the form $\exp (cx_1)$ such that $S_0^{-1}\tilde{q}_{l'}S_0=\partial_1$ (here $\tilde{q}_{l'}$ is a linear polynomial with constant coefficients). Then obviously the operator $S'=SS_0^{-1}$ don't depend on $x_1$. So, $S=S'S_0$. 

From remark \ref{alpha=1} we know that $\tilde{P}, \tilde{Q}\in \Pi$ and from item i) we know that $\tilde{p}_{k'}, \tilde{p}_{k'-1}$ are operators with constant coefficients (and $\tilde{p}_{k'}=\partial_2^{k'}$). Thus, $S'$ has the form
$$
S'=\exp (F(x_2, \partial_1)),
$$
where $F$ is a polynomial in $x_2, \partial_1$. This polynomial is linear iff $\tilde{p}_{k'-1}$ is linear. But if it is not linear, then the operator $(S')^{-1}\tilde{P}S'$ will not satisfy the condition $A_1$ (as $\tilde{P}$ satisfies $A_1$ for some $(k,l)$), a contradiction. So, it is linear and we are done. 
\end{proof}

\begin{nt}
\label{as_this_lemma}
As this lemma shows, if there is a pair of normalized operators in $B$, then any equivalent ring $B'$ that has a pair of normalized operators is obtained from $B$ by conjugation with an operator of special form, and this conjugation is equivalent to a linear change of variables 
\begin{equation}
\label{special_change}
\partial_2\mapsto \partial_2+c\partial_1 +b, \mbox{\quad} \partial_1\mapsto \partial_1+d
\end{equation}
with $c,b,d\in k$. The Schur pair corresponding to such a ring $B'$ will be equivalent to the first one as well. 

Conversely, if one starts from any Schur pair $(A,W)$ in a given equivalence class, then the ring $B$ can be constructed as $B=SAS^{-1}$, where $S$ now comes from the analogue of the Sato theorem (see theorem \ref{Sato} below). If $(A',W')$ is an equivalent Schur pair, then $A'=T^{-1}AT$, $W'=WT$ for some $1$-admissible operator $T$, which can be written (see remark \ref{admnote}) in the form $T=T'T_0$, where 
$T_0$ has the form \eqref{special_form}, and $T'=1+T^-$, where $T^-\in \hat{D}_1[[\partial_2^{-1}]]\partial_2^{-1}$. Then it is easy to see that the corresponding Sato operator for the space $W'$ from theorem \ref{Sato} is $S'=T_0^{-1}ST'T_0$. So, the corresponding ring $B'=S'A'(S')^{-1}=T_0^{-1}BT_0$, i.e. it is obtained from $B$ by  the linear change \eqref{special_change}. It will automatically contain a  pair of normalized operators. 

To find a pair of normalized operators in a given ring $B$ we need sometimes to replace $B$ by an equivalent ring (see \cite[Lemma 2.10]{Zhe2}). 
\end{nt}

\begin{theo}{(\cite[Th.3.1]{Zhe2})}
\label{Sato}
Let $W$ be a $k$-subspace 
$W\subset  k[z_1^{-1}]((z_2))$ 
with $\Sup (W)=W_0$.
Let $\{w_{i,j}, i,j\ge 0\}$ be the unique basis in $W$ with the property $w_{i,j}=z_1^{-i}z_2^{-j}+w_{i,j}^-$, where $w_{i,j}^-\in k[z_1^{-1}][[z_2]]z_2$. Assume that all elements $w_{i,j}$ satisfy the condition $A_{1}$. 

Then there exists a unique operator $S=1+S^-$ satisfying $A_{1}$, where $S^-\in \hat{D}_1[[\partial_2^{-1}]]\partial_2^{-1}$, such that 
$W_0S=W $. 
\end{theo}

The Schur pairs from theorem \ref{schurpair} one to one correspond to pairs of subspaces in the space $k[[u]]((t))$ via an isomorphism 
\begin{equation}
\label{psi_1}
\psi_1:k[z_1^{-1}]((z_2))\cap \Pi\simeq k[[u]]((t)) \mbox{\quad } z_2\mapsto t, z_1^{-1}\mapsto ut^{-1},
\end{equation}
where $k[z_1^{-1}]((z_2))\cap \Pi$ denotes the $k$-subspace generated by series satisfying the condition $A_1$ (see \cite[Cor.3.3]{Zhe2}). We will denote these pairs by the same letters $(A,W)$. Obviously, $W\cdot A\subset W$. 

\begin{defin}{(\cite[Def.3.5,3.6]{Zhe2})}
\label{goodring}
For the ring $A\subset k[[u]]((t))$  define 
$$
N_A=GCD\{\nu_t(a), \mbox{\quad} a\in A \mbox{\quad such that \quad} \nu (a)=(0,*)\},
$$ 
where $*$ means any value of the valuation. Define 
$$
\tilde{N}_A=GCD\{\nu_t(a), \mbox{\quad} a\in A\}. 
$$ 
We'll say that the ring $A$ is strongly admissible if there is an element $a\in A$ with $\nu (a)=(1,*)$ and $\tilde{N}_A=N_A$. 
\end{defin}

\begin{defin}{(\cite[Def.3.8]{Zhe2})}
\label{ggg1}
For $1$-quasi elliptic commutative ring $B\subset \hat{D}$ we define numbers $\tilde{N}_B$, $N_B$ to be equal to the numbers $\tilde{N}_A$, $N_A$, where $A$ is the ring corresponding to $B$ by theorem \ref{schurpair} (after applying the isomorphism \eqref{psi_1}). We say that $B$ is strongly admissible if $A$ is strongly admissible. 

For a strongly admissible ring $B$ we define the rank of $B$ as 
$$
\rk (B)=N_B=\tilde{N}_B
$$
\end{defin}

\begin{nt}
\label{rank_for_PDO}
If the ring $B$ in the definition \ref{ggg1} is a ring of PDO's, then the numbers $\tilde{N}_B$, $N_B$ and the rank can be defined similarly to definition \ref{goodring}:
$$
N_B=GCD\{\ord (Q), \mbox{\quad} Q\in B \mbox{\quad such that \quad} \ord_{\Gamma} (Q)=(0,*)\}.
$$ 
 Define 
$$
\tilde{N}_B=GCD\{\Ord (Q), \mbox{\quad} Q\in B\},
$$ 
where $\Ord$ means the usual order in the ring $D$.

Analogously, $B$ is strongly admissible if there is an element $Q\in B$ with $\ord_{\Gamma} (Q)=(1,*)$ and $\tilde{N}_B=N_B$. Its rank $\rk (B)=N_B=\tilde{N}_B$. 

We would like to emphasize that the rank of the ring $B$ defined as $N_B=\tilde{N}_B$  is less or equal to the rank of the sheaf of common eigenfunctions of the operators  from $B$  (cf. \cite[Rem.2.3]{Ku3}, this notion of rank is often used in various papers).   This follows from \cite[Prop. 3.3, Prop. 3.2, Th.2.1]{Ku3}. 

Below we will write $\Rk (B)$ to denote the rank in the second sense.  
\end{nt}

\begin{theo}{(\cite[Th.3.4]{Zhe2})}
\label{dannye2}
There is a one to one correspondence between the set of classes of equivalent $1$-quasi elliptic strongly admissible finitely generated rings of operators in $\hat{D}\cap \Pi$ of rank $r$ and the set of isomorphism classes of geometric data $\cm_r$ of rank $r$. 
\end{theo}

If we have a ring $B\subset D$ of commuting PDOs satisfying the property \ref{property}, then by \cite[Lemma 2.6]{Zhe2} and by \cite[Prop.2.4]{Zhe2} (cf. also the beginning of section 3.1 in loc. cit.) there is a linear change of variables making this ring $1$-quasi elliptic strongly admissible. 
Moreover, as it follows from the proofs of \cite[Lemma 2.6, Prop.2.4]{Zhe2}, almost all linear changes of variables preserve the property of the ring to be $1$-quasi elliptic strongly admissible. In particular, for almost all linear changes we have the following extra property of the operators $P,Q$ from definition \ref{elliptic}:
\begin{equation}
\label{extraproperty}
\sigma (P)=\xi_2^k+\sum_{q=1}^{k}h_q\xi_1^q\xi_2^{k-q}, \mbox{\quad } h_k\neq 0; \mbox{\quad} \sigma (Q)=\xi_1\xi_2^l+\sum_{q=2}^{l+1}c_q\xi_1^q\xi_2^{l+1-q}, \mbox{\quad } c_{l+1}\neq 0.
\end{equation}

\begin{nt}
\label{remark}
From the construction in 
\cite[Sec.3]{Zhe2} explaining the correspondence between  geometric data and $1$-quasi elliptic strongly admissible rings it follows that the ring after such linear change of variables corresponds to the data with the same surface and divisor, but with probably other sheaf, other point $P$ and trivializations $\pi ,\phi$ (cf. also remark \ref{nt} and \cite[Th.2.1, Prop.3.3]{Ku3}).  
\end{nt}

\begin{nt}
\label{B_of_rank_r}
If the ring $B$ of PDOs is $1$-quasi-elliptic strongly admissible, then, obviously, there exist two operators $P,Q$ as in definition \ref{elliptic} with $k=l+1=\Ord (P)$. In this situation, examining the arguments from the proof of Lemma 2.10, item 1 in \cite{Zhe2}, we see that there are some $\beta\in k$ and $f\in k[[x_1,x_2]]^*$ such that the operators $f^{-1}(P+\beta Q)f$, $f^{-1}Qf$ are normalized (in the sense of \cite[Def. 2.19]{Zhe2}). Thus, in the equivalent class of $B$ we can find a ring of PDOs with a pair of normalized operators. 

As the arguments from remark \ref{as_this_lemma} show, any Schur pair equivalent to the Schur pair associated with $B$ leads to a ring $B'$ obtained from $B$ by the linear change of variables \eqref{special_change}. Thus, $B'$ is also a ring of PDOs! 
\end{nt}

There is another nice property of $1$-quasi elliptic subrings of partial differential operators claiming the "purity" of such rings:
\begin{prop}{(\cite[Prop.3.1]{Zhe2})}
\label{purity}
Let $B\subset D\subset \hat{D}$ be a $1$-quasi elliptic ring of commuting partial differential operators. Then any ring $B'\subset \hat{D}$ of commuting operators such that $B'\supset B$ is a ring of partial differential operators, i.e. $B'\subset D$.
\end{prop} 

\subsection{Schur pairs, ribbons and moduli space}
\label{sec.2.3}

We would like to recall that in the classical KP theory there are well known geometric data classifying the commutative rings of ordinary differential operators. These data consist of a projective curve over a field $k$ plus a line bundle (or a  torsion free sheave if the curve is singular) plus some additional data (a distinguished point $p$ of the curve
plus a formal local parameter at $p$ , and a formal trivialization at $p$ of the sheaf).  
Also there is a map which associate to each such data a pair of subspaces $(A,W)$ ("Schur pair") in the space $V=k((z))$, where $A\varsupsetneq k$ is a stabilizer $k$-subalgebra of $W$ in $V$: $A\cdot W\subset W$, and $W$ is a point of the infinite-dimensional Sato grassmannian (see e.g. \cite{Mul} for details).
 This map is usually called as the Krichever  map in the literature. In works \cite{Pa1}, \cite{Pa} (see also \cite{Os}) Parshin introduced an analogue of the Krichever map which associates to each geometric data (which include in that works a Cohen-Macaulay surface, an ample Cartier divisor, a smooth point
and a vector bundle) a pair of subspaces $(\mathbb{A},\mathbb{W})$ in the two-dimensional
local field associated with the flag (surface, divisor, point) $k((u))((t))$ (with analogous properties). He showed that this map is injective on such data. In works \cite{Pa}, \cite{Os} some combinatorial
construction was also given. This construction helps to calculate cohomology groups of vector bundles in
terms of these subspaces and permits to reconstruct the geometric data from the pair $(\mathbb{A},\mathbb{W})$. The difference of this new Krichever-Parshin map from the Krichever map is that the last map is known to be bijective.

To extend the Krichever-Parshin map and to make it bijective we introduced in the work \cite{Ku} new geometric objects called formal punctured ribbons (or simply ribbons for short) and torsion free coherent sheaves on them, we extended this map on the set of new geometric data which include these objects and showed the bijection between the set of geometric data and the set of pairs of subspaces $(\mathbb{A},\mathbb{W})$ (also called  generalized Schur pairs) satisfying certain combinatorial conditions. We also showed that for any given Parshin's geometric data one can construct a unique geometric data with a ribbon, and the initial Parshin's data can be reconstructed from the new data with help of the combinatorial construction mentioned above. In the work \cite{Ku3} we extended this construction to the modified Parshin's data from definition \ref{geomdata}. 

\subsubsection{Properties of the Krichever map in dimension one}
\label{svva}

First let's recall some properties of the classical Krichever map for torsion free sheaves of rank one (see \cite{Mu}, \cite{Mul} for details; we change slightly some notation from these papers here). 
In this case the geometric data is a quintet $(C,P,\cf , u, \phi )$, where $C$ is a projective curve, $P$ is a smooth point on $C$, $\cf$ is a torsion free sheaf of rank one, $u$ is a local parameter at the point $P$ (in particular, there is an isomorphism $\pi:\hat{\co}_{C,P}\simeq k[[u]]$), and $\phi :\hat{\cf}_P\simeq k[[u]]$ is a trivialisation. Obviously, any two such trivialisations differ by multiplication on an element $a\in k[[u]]^*$. The Schur pair $A,W\subset k((u))$ is constructed in  analogous way as in section \ref{data}: $A$ is the image of the group $H^0(C\backslash P,\co_C)$ in the space $k((u))$ (which is obtained using the trivialisation $\pi$), and $W$ is the image of the group $H^0(C\backslash P,\cf )$ in the space $k((u))$ (which is obtained using the trivialisation $\phi$). 
If we choose another trivialisation of the sheaf $\cf$ then the new space $W$ will differ from the old one  by multiplication on the element $a\in k[[u]]^*$ and the space $A$ will not change. 

In this case we also have the following  properties  in terms of spaces $A,W$: 
\begin{equation}
\label{kriv1}
\cf (nP)\simeq \Proj (\tilde{W}(n)), 
\end{equation}
where $\tilde{W}=\oplus_{n=0}^{\infty} W_n s^n$, $W_n=W\cap u^n\cdot k[[u]]$; 
\begin{equation}
\label{kriv2}
H^0(C,\cf )\simeq W\cap k[[u]], \mbox{\quad} H^1(C,\cf )\simeq \frac{k((u))}{W+k[[u]]}.
\end{equation}
Recall that all torsion free rank one sheaves with fixed Euler characteristic can be divided into the union of orbits of the group $\Pic^0(C)$. Namely (see \cite[Sec.6]{SW}) there are maximal torsion free sheaves, i.e. sheaves not isomorphic to a direct image of a torsion free sheaf on a (partial) normalisation of the curve $C$, and not maximal torsion free sheaves, i.e. sheaves isomorphic to direct images of torsion free sheaves on partial normalisations of $C$. If the sheaf is maximal then the action of the group $\Pic^0(C)$ is free on it.  Thus, every orbit of a torsion free sheaf of rank one is a torsor over $\Pic^0(C')$, where $C'$ is a partial normalisation of $C$. This torsor has a topology induced by the topology of $\Pic^0(C')$. 

Further we will need the following fact: if the Euler characteristic of a sheaf $\cf$ on a projective curve $C$ is $k\ge 0$, then there exists a dense open subset $V$ in the orbit of this sheaf such that for each $\cl\in V$ 
\begin{equation}
\label{kriv3}
H^1(C,\cl )=0, \mbox{\quad} h^0(C,\cl )=k, \mbox{\quad} H^1(C, \cl (-kP))=0, \mbox{\quad} H^0(C, \cl (-kP))=0
\end{equation}
for some point $P\in C$. 
This fact can be proved by induction on $k$ as follows. For any torsion free sheaf $\cf$ with Euler characteristic $k\ge 0$, any fixed smooth point $Q$ and $n\gg 0$ we have $H^1(C, \cf (nQ))=0$, $h^0(C,\cf (nQ))=k+n >0$. For any fixed  global section $a\in H^0(C,\cf (nQ))$ there is a dense open subset $U\subset C$ such that for any $P\in U$ the image of an embedding of $a$ in $\co_{C,P}$ (with help of any trivialisation) is invertible. Then from properties \eqref{kriv1} and \eqref{kriv2} it follows that 
$$
h^0(C,\cf (nQ-P))=n+k-1, \mbox{\quad} H^1(C, \cf (nQ-P))=0. 
$$
Thus, by induction there exists an open subset $U'\subset C$ such that for any $P\in U'$
$$
H^0(C,\cf (nQ-(n+k)P))=0, \mbox{\quad} H^1(C,\cf (nQ-(n+k)P))=0, 
$$
$$
\mbox{\quad} H^1(C,\cf (n(Q-P)))=0, \mbox{\quad} h^0(C,\cf (n(Q-P)))=k.
$$
The rest of the proof follows from \cite[Th.12.8]{Ha} for the morphism $f: Pic^0(C)\times C\rightarrow Pic^0(C)$ and the Poincare sheaf $\cp$. 

\subsubsection{Three properties of the pair $(\mathbb{A},\mathbb{W})$}

Now we would like to recall three properties of the pair $(\mathbb{A},\mathbb{W})$. First we recall (see, for example,~\cite{Ku}) that a
$k$-subspace $W$ in $k((u))$ is called a {\em Fredholm}
subspace if
$$
\dim_k W \cap k[[u]] < \infty  \qquad \mbox{and }\qquad
\dim_k \frac{k((u))}{W + k[[u]]} < \infty
\mbox{.}
$$

For a $k$-subspace $W$ in $k((u))((t))$, for  $n \in \dz$
let
$$
W(n) = \frac{W \cap t^nk((u))[[t]]}{W \cap
t^{n+1}k((u))[[t]]}
$$
be a $k$-subspace in $k((u)) =
\frac{t^nk((u))[[t]]}{t^{n+1}k((u))[[t]]}$.

\subsubsection{The first property}
 Let the pair $(\mathbb{A},\mathbb{W})$ be the image of the ribbon's data corresponding to some geometric data of rank one from definition \ref{geomdata} with $\cf$ being coherent of rank one (for details see \cite[Sec.3.5]{Ku3}). 
Recall (see definition \ref{zeta}, remark \ref{torsion_free_sheaves}) that for such rank 1 torsion free sheaves the map $\zeta$ \eqref{zeta} was defined. 
Then (see the proof of theorem 1 in \cite{Ku})
\begin{equation}
\label{a-property}
\da (nd)\simeq \mbox{the image of the quintet $(C,P,\co_C(nC'),u,id)$ in $k((u))$ under the Krichever map},
\end{equation}
where $C'=dC$ is the ample Cartier divisor as above (note that $\co_C(nC')\simeq \zeta (\co_X(nC'))$), and 
\begin{equation}
\label{w-property}
\dw (nd+k)\simeq \mbox{the image of the quintet $(C,P,\zeta (\cf_k(nC')),u,\phi )$  under the Krichever map},
\end{equation}
where $0\le k< d$ and $\phi$ is some trivialization of the sheaf $\zeta (\cf_k(nC'))$ at the point $P$ on $C$ (note that $\zeta (\cf_k(nC'))\simeq (\cf_k/\cf_{k-1})(nC')$). 
Thus, from one-dimensional KP theory (see \eqref{kriv2}) we have
$$
H^0(C,(\cf_k/\cf_{k-1})(nC'))\simeq \dw (nd+k)\cap k[[u]],  
$$
\begin{equation}
\label{1-KP}
H^1(C,(\cf_k/\cf_{k-1})(nC'))\simeq k((u))/(\dw (nd+k)+k[[u]])
\end{equation}

\subsubsection{The second property} Assume that the pair $\da ,\dw \in k((u))((t))$ comes from a geometric data of rank one. Then
\begin{equation}
\label{picture_cohomology0}
H^0(X,\co_X(nC'))\simeq  \da \cdot t^{nd} \cap k[[u]]((t)) \cap k((u))[[t]] \mbox{,}
\end{equation}
\begin{equation}
\label{picture_cohomology1}
H^1(X,\co_X(nC'))\simeq  \frac{\da \cdot t^{nd} \cap (k[[u]]((t)) + k((u))[[t]])} {\da \cdot t^{nd} \cap
k[[u]]((t)) + \da \cdot t^{nd} \cap k((u))[[t]]} \mbox{,}
\end{equation}
\begin{equation}
\label{picture_cohomology2}
H^2(X,\co_X(nC'))\simeq  \frac{k((u))((t))}{\da \cdot t^{nd} + k[[u]]((t)) + k((u))[[t]]} \mbox{.}
\end{equation}
For the proof see remark 3 and lemma 1 from \cite{Ku1} (remark 3 refers for the proof to papers \cite{Os,
Pa}, where $C$ was assumed to be a Cartier divisor; in general case it is not difficult to improve the proof from these papers; nevertheless, we will need this property in our paper only for such cases when $C$ is known to be Cartier). In particular, if $C$ is a Cartier divisor, it follows that 
\begin{equation}
\label{car-property}
\co_X(nC)\simeq \co_{X,n}, \mbox{\quad} \zeta (\co_X(nC))\simeq \co_C(nC)
\end{equation} 
for any $n$ (cf. remark \ref{torsion_free_sheaves}).

\subsubsection{The third property}: 
\begin{equation}
\label{intersection}
\mbox{If $A$ is Cohen-Macaulay ring then $A=\mathbb{A}\cap k[[u]]((t))$}, 
\end{equation}
where $A,W$ are the subspaces in $k[[u]]((t))$ constructed above starting from the geometric data. This claim was proved in \cite[Remark 3.4]{Ku3}. In the introduction to the loc.cit. the analogous property for the space $W$ was also announced (though imprecise): 
$$
W=\mathbb{W}\cap k[[u]]((t)). 
$$
We are going to clarify it here. 

\begin{prop}
\label{zero_cohomology}
Let $\cf$ be a coherent rank one torsion free sheaf on a projective surface $X$ defined over an uncountable algebraically closed field $k$. Assume that there is an ample irreducible $\dq$-Cartier divisor $C\subset X$ not contained in the singular locus such that $C^2=1$. Let $C'=dC$ be a very ample Cartier divisor. 
Suppose that the following conditions hold (see remark \ref{key}, definition \ref{zeta}): 
$$
\chi (X,\cf (nC'))=\frac{(nd+1)(nd+2)}{2}, \mbox{\quad} H^0(C,\zeta (\cf_k )(-(k+1)Q))=H^1(C, \zeta (\cf_k )(-(k+1)Q))=0
$$ 
for a smooth point $Q\in C$, $n\ge 0$, where $0\le k <d$. 
Then 

i) there exists a point $P\in C$ regular in $C$ and $X$ such that 
the conditions from item 6 of definition \ref{geomdata} hold for a trivialisation $\hat{\phi} :\hat{\cf}_P\simeq k[[u,t]]$, i.e. the homomorphisms
$$
H^0(X,\cf (nC'))\rightarrow k[[u,t]]/(u,t)^{nd+1}
$$
are isomorphisms for all $n\ge 0$;\footnote{For reader who prefer the alternative definition this item can be reformulated as follows. There exist $j:T\rightarrow X$ with $j(O)=P\in C\backslash (C^{sing}\cup X^{sing})$ and $j^{-1}(C)=T_1$ (i.e. $R=k[[u,t]]\simeq \hat{\co}_{X,P}$, $R/tR\simeq \hat{\co}_{C,P}$) such that for a choice of generator of $\cf_P$ as $\co_{X,P}$-module the embedding $\cf \hookrightarrow j_*\co_T$ (corresponding to $(j^*\cf )_O=R\otimes_{\co_{X,P}}\cf \simeq \hat{\cf}_P$) condition \ref{GD3} is satisfied, i.e. the maps $H^0(X,\cf (nC'))\rightarrow R\cdot t^{-nd}/\cm^{nd+1}\cdot t^{-nd}$ are isomorphisms. }

ii) for this trivialisation 
$$
W=\mathbb{W}\cap k[[u]]((t));
$$

iii) for this trivialisation 
\begin{equation}
\label{picture_cohomology1.1}
H^1(X,\cf )\simeq  \frac{\dw  \cap (k[[u]]((t)) + k((u))[[t]])} {\dw  \cap
k[[u]]((t)) + \dw  \cap k((u))[[t]]} = 0 \mbox{,}
\end{equation}
\begin{equation}
\label{picture_cohomology2.1}
H^2(X,\cf )\simeq  \frac{k((u))((t))}{\dw  + k[[u]]((t)) + k((u))[[t]]} =0 \mbox{;}
\end{equation}

iv) the sheaf $\cf$ is Cohen-Macaulay on $X$. 

\end{prop}

\begin{proof}
i) For any sheaf $\cf_k$ and $m>0$ we have the short exact sequences 
$$
0\rightarrow \zeta (\cf_k) \rightarrow \zeta (\cf_k)\otimes_{\co_C}\co_X(mC')|_C\rightarrow \zeta (\cf_k) \otimes_{\co_C}(\co_X(mC')|_C/\co_C)\rightarrow 0,
$$
since $\co_X(mC')|_C$ is an invertible sheaf. Hence we have $H^1(C,\zeta(\cf_k)\otimes_{\co_C}\co_X(mC')|_C)=0$ for all $m\ge 0$. Since $C^2=1$ (i.e. $\deg (\co_X(C')|_C)=d$), by the asymptotic Riemann-Roch theorem we have 
\begin{equation}
\label{***}
\chi (\zeta (\cf_k)\otimes_{\co_C}\co_X(mC')|_C)=h^0(C,\zeta (\cf_k)\otimes_{\co_C}\co_X(mC')|_C)=md+k+1.
\end{equation}
For each $m\ge 0$ by the property \eqref{kriv3} there is an open subset $U_m$ in $C$ such that 
\begin{multline}
\label{****}
H^0(C,\zeta (\cf_k)\otimes_{\co_C}\co_X(mC')|_C\otimes_{\co_C}\co_C(-(md+k+1)P))= \\
H^1(C,\zeta (\cf_k)\otimes_{\co_C}\co_X(mC')|_C\otimes_{\co_C}\co_C(-(md+k+1)P))=0 
\end{multline}
for any $P\in U_m$. Therefore, since the ground field is uncountable, there exists a point $P\in  \cap_{m=0}^{\infty}U_m$ regular in $C$ and $X$ such that these properties hold simultaneously for all $m\ge 0$ and $0\le k<d$. 
Since for any $n\ge 0$ by lemma \ref{claim1} we have 
$$
W_{nd+k}/W_{nd+k-1}\simeq H^0(X,\cf_k(nC'))/H^0(X,\cf_{k-1}(nC')) \hookrightarrow H^0(C,\zeta (\cf_k)\otimes_{\co_C}\co_X(nC')|_C),
$$
and $\chi_1(H^0(X, \cf (nC')))\subset W_{nd}$ by definition, we have that for all $n\gg 0$ 
\begin{multline}
h^0(X, \cf (nC'))=\chi (\cf (nC'))=\frac{(nd+1)(nd+2)}{2} \le \dim_k(W_{nd}) \le \\
\sum_{k=0}^{d-1}\sum_{m=0}^{n-1} h^0(C,\zeta (\cf_k)\otimes_{\co_C}\co_X(mC')|_C)+h^0(C,\zeta (\cf_k)\otimes_{\co_C}\co_X(nC')|_C).
\end{multline}
By \eqref{***} these inequalities are equalities. Therefore, $H^0(X, \cf (nC'))\simeq W_{nd}$ for any $n\gg 0$. Hence  $\cf \simeq  \cf_0$  and 
\begin{equation}
\label{5*}
H^0(X, \cf (nC'))\simeq W_{nd}, \mbox{\quad } W_{nd+k}/W_{nd+k-1} \simeq H^0(C,\zeta (\cf_k)\otimes_{\co_C}\co_X(nC')|_C)
\end{equation}
for all $n\ge 0$ by remark \ref{torsion_free_sheaves} and by lemma \ref{claim1}. 
Together with \eqref{****} this implies that the conditions from item 6 of definition \ref{geomdata} hold for some trivialisation at the point $P$. 

ii) By \eqref{kriv2} and \eqref{w-property} we have 
$$
H^0(C,\zeta (\cf_k)\otimes_{\co_C}\co_X(nC')|_C)\simeq \dw (nd+k)\cap k[[u]].
$$
From this and from \eqref{5*} follows that $W=\dw \cap k[[u]]((t))$. 

iii) By assumption on the Euler characteristic of the sheaf $\cf$ and from \eqref{5*} it follows $h^1(X,\cf )-h^2(X,\cf )=0$. From \eqref{****} we know that $h^1(C,\zeta (\cf_k)\otimes_{\co_C}\co_X(nC')|_C)=0$ for any $0\le k<d$ and $n\ge 0$. Recall that we have exact sequences
$$
0\rightarrow \cf_k\rightarrow \cf_{k+1}\rightarrow \zeta (\cf_{k+1})\rightarrow 0
$$
for any $0\le k$. So, from the induced long exact cohomological sequences and from \eqref{5*} we obtain 
$H^1(X,\cf_k)\simeq H^1(X,\cf_{k+1})$ for any $k\ge 0$. Thus, all these groups are zero, since $H^1(X, \cf_{k+nd})=H^1(X, \cf_k(nC'))=0$ for all $n\gg 0$. Hence $H^2(X, \cf )=0$. From item ii) we get
$$
\dw \cap (k[[u]]((t))+k((u))[[t]])\subset \dw \cap k[[u]]((t))+\dw \cap k((u))[[t]],
$$
where from 
$$
\frac{\dw  \cap (k[[u]]((t)) + k((u))[[t]])} {\dw  \cap
k[[u]]((t)) + \dw  \cap k((u))[[t]]} = 0.
$$
By \eqref{kriv2} and \eqref{w-property} we have 
$$
0=H^1(C,\zeta (\cf_k)\otimes_{\co_C}\co_X(nC')|_C)\simeq \frac{k((u))}{\dw +k[[u]]}
$$
for all $k\ge 0$. 
Hence
$$
\frac{k((u))((t))}{\dw  + k[[u]]((t)) + k((u))[[t]]} =0.
$$

iv) By \cite[Cor.3.1]{Ku3} the sheaf $\cf$ is Cohen-Macaulay along $C$. The same arguments show that the sheaves $\cf_k$ are Cohen-Macaulay along $C$. 
Consider the Macaulaysation $CM(\cf )$ of the sheaf $\cf$ (see \cite[Appendix B]{Ku3}). 
Consider the image $W'=\chi_1(H^0(X\backslash C,CM(\cf )))$ in $k[[u]]((t))$, where we use the same embedding $\phi$ of the sheaf $\cf$ to define $\chi_1$ (cf. section \ref{data}). 
Let's note that $CM(\cf )|_{mC'}=\cf |_{mC'}$ is pure of dimension one for any $m>0$, since $\cf$ is Cohen-Macaulay along $C$. Then by lemma \ref{claim1} we have $H^0(X, CM(\cf )(nC'))\simeq W'_{nd}$ for all $n\ge 0$.  

Directly from definition of a Cohen-Macaulay sheaf follows that the sheaves $CM(\cf )_k$ are Cohen-Macaulay for all $k$. Note that $CM(\cf_k)\simeq CM(\cf )_k$. Indeed, by definition of Cohen-Macaulaysation we have $CM(\cf_k) \subset CM(\cf )_k$. If $CM(\cf_k) \not\simeq CM(\cf )_k$, then this would mean 
$CM(\cf_k)_{-k}\not\simeq (CM(\cf )_k)_{-k}\simeq CM(\cf )$. But $CM(\cf_k)_{-k}\simeq CM(\cf )$, since $CM(\cf_k)_{-k}\subset (CM(\cf )_k)_{-k}\simeq CM(\cf )$ and $CM(\cf_k)_{-k}$ is a Cohen-Macaulay sheaf containing $\cf$ (cf. \cite[Rem.B.2]{Ku3}). 

In particular, we can apply the construction from \cite[Sec.3.5]{Ku3} and construct a torsion free sheaf $\cn$ on the ribbon $(C,\ca )$ (the ribbon constructed by our geometric data). Then we can construct a space $\dw'\subset k((u))((t))$ by the sheaf $\cn$. Since the construction depends only on sections of the sheaves $CM(\cf_k)$ along the curve $C$, we get $\dw'=\dw$. Hence from item ii) we obtain $\cf\simeq CM(\cf )$. 

\end{proof}

\begin{nt}
\label{moduli_space}
Torsion free sheaves of rank one on the projective surface $X$ with fixed Hilbert polynomial  $\chi$ from proposition \ref{zero_cohomology} are stable in the sense of standard definition from \cite[Ch.2]{HL}. Stable sheaves are parametrized by a projective scheme $\cm_X(\chi )$ (see Chapter 4 in loc.cit.). 

On the other hand, all sheaves we are interested in satisfy the assumptions of lemma \ref{claim2} (and in view of theorems \ref{CMmodules} and \ref{trivial} even more strong assumption: they are Cohen-Macaulay on $X$). By \cite[Prop.1.2.16]{BH} the Cohen-Macaulayness is an open condition.  
So, it is reasonable to consider an open subscheme $\cm_X^1$  of the moduli space $\cm_X(\chi )$ parametrising such sheaves.
Then the map $\zeta$ from section \ref{mapzeta} induces the morphism 
$$
\zeta :\cm_X^1\rightarrow \cm_C(g),
$$
where $\cm_C(g)$ is the moduli space of rank one torsion free sheaves of degree $g=p_a(C)$ on $C$ (cf. \cite{Rego}). We conjecture that this morphism is surjective (cf. examples at the end of this paper).
\end{nt}

\section{Theorems}

\subsection{Cohen-Macaulay property for PDOs}

Recall that by \cite[Th.2.1]{Ku3} any commutative ring $B\subset D$ of PDOs satisfying the property \eqref{property} leads to a datum $(X,C,\cl )$, where $X,C$ are the same as in definition \ref{geomdata} and $\cl$ is a torsion free coherent sheaf on $X$. Let's assume (see the discussion before remark \ref{nt}) that $B$ is a 1-quasi-elliptic strongly admissible ring satisfying the extra property \eqref{extraproperty}. 

In this case by \cite[Prop. 3.3]{Ku3} the datum $(X,C,\cl )$ is isomorphic to the triple $(X,C,\cf )$ (a part of geometric data) from theorem \ref{dannye2}. If $B$ is of rank one, then by \cite[Th. 2.1]{Ku3} we have $C^2=1$, and by \cite[Prop. 3.2]{Ku3} the sheaf $\cf$ and the geometric data from theorem \ref{dannye2} are of rank one. 

\begin{theo}
\label{CMmodules}
Let $(X,C,P,\cf ,\pi ,\phi )$ be a geometric datum  corresponding to a  1-quasi-elliptic strongly admissible ring $B\subset D$ of commuting operators satisfying the properties \eqref{property}, \eqref{extraproperty}. 

Then $\cf$ is a Cohen-Macaulay sheaf on $X$. 
\end{theo}

\begin{nt}
\label{prosto}
If the ring $B$ is maximal then  by \cite[Th.4.1]{Zhe2} the surface $X$ is also Cohen-Macaulay. 
\end{nt}

\begin{proof}
By \cite[Prop.3.2]{Ku3} the sheaf $\cf \simeq \cl$ is coherent. By \cite[Th.2.1, Prop. 3.3]{Ku3} and remark \ref{Kur} the rank of the geometric datum is one. 

Consider the Macaulaysation $CM(\cf )$ of the sheaf $\cf$ (see \cite[Appendix B]{Ku3}). By \cite[Cor. 3.1]{Ku3} the sheaf $\cf$ is Cohen-Macaulay along $C$; in particular, $\cf_P\simeq CM(\cf )_P$.  
Consider the image $W'=\chi_1(H^0(X\backslash C,CM(\cf )))$ in $k[[u]]((t))$, where we use the embedding $\phi$ of the sheaf $\cf$ to define $\chi_1$ (cf. section \ref{data}). Then we claim that this image is a finitely generated linear space over $W=\chi_1(H^0(X\backslash C, \cf ))$:
$$
W'=\langle W, w_1,\ldots ,w_k\rangle ,
$$
where $w_1,\ldots ,w_k\notin W$, $w_1,\ldots ,w_k\in k[[u]]((t))$. 

To prove the claim, first of all let's note that the sheaf $CM(\cf )|_{mC'}=\cf |_{mC'}$ is pure of dimension one for any $m>0$ (see remark  \ref{torsion_free_sheaves} for notation), since $\cf$ is Cohen-Macaulay along $C$. 

Let's show that the condition on the space $W'$ from lemma \ref{claim1} is satisfied. This condition is true for elements $w$ from the space $W$ corresponding to the sheaf $\cf$, because $W_{nd}\simeq H^0(X,\cf (nC'))$. It is also clear that for any element $w$ from $W'$ we have $w\in f^{-md}\cf_P$ for some $m$. Now take any element $w\in W_{nd}'$ Then we have for all sufficiently big $m>0$: 
$$
f^{-md}w-c_1w_1-\ldots -c_kw_k=a\in W_{(n+m)d}
$$
for some $c_1,\ldots ,c_k\in k$. Thus, $f^{nd}w= f^{(n+m)d}a+f^{(n+m)d}(c_1w_1+\ldots +c_kw_k)\in \cf_P$.

Now by lemma \ref{claim1} we have $H^0(X, CM(\cf )(nC'))\simeq W'_{nd}$ for all $n\ge 0$. 

As a corollary we get that for $n>0$ big enough 
$$
W'_{nd}/W'_{(n-1)d}\simeq H^0(C, CM(\cf )(nC')|_{C'})=H^0(C, \cf (nC')|_{C'})\simeq W_{nd}/W_{(n-1)d}.
$$
Obviously, $W'_{nd}\supset W_{nd}$ for all $n$. So, our claim is proved. 

By \cite[Th.3.3, Th.3.1]{Zhe2} there is a unique operator $S$ satisfying condition $A_1$ such that $\psi_1^{-1}(W)=W_0S$ (the map $\psi_1$ is defined in \eqref{psi_1}), where $W_0=k[z_1^{-1},z_2^{-1}]$. Moreover, $B=S\psi_1^{-1}(A)S^{-1}$, where $A=\chi_1(H^0(X\backslash C, \co_X))$. Since $W'\cdot A\subset W'$, we have 
$$
(\psi_1^{-1}(W')S^{-1}) \cdot B \subset (\psi_1^{-1}(W')S^{-1}), \mbox{\quad } \psi_1^{-1}(W')S^{-1} =\langle W_0, \tilde{w}_1, \ldots ,\tilde{w}_k\rangle ,
$$ 
where $\tilde{w}_i=\psi_1^{-1}(w_i)S^{-1}$. Each series $\tilde{w}_i$ can be written in the following way:
$$
\tilde{w}_i=w'_i+w''_i, \mbox{\quad } w'_i=\sum_{k\ge 0, l>0, k+l=q_i} c_{kl}z_1^{-k}z_2^{l},  \mbox{\quad } w''_i=\sum_{k\ge 0, l>0, k+l<q_i} b_{kl}z_1^{-k}z_2^{l}
$$
To the end of this proof we will call $q_i$ the order of $w'_i$: $\ord (w'_i)=q_i$. Since the ring $B$ satisfies the property \eqref{extraproperty}, for any $n> 0$ the symbols of the operators $P^n$, $Q^n$ satisfy the same property \eqref{extraproperty} with $k,l$ replaced by $kn,ln$. For all $n\gg 0$ and for any $\tilde{w}_i$ we must have 
$$
\tilde{w}_i P^n \in (\psi_1^{-1}(W')S^{-1}), \mbox{\quad} \tilde{w}_i Q^n \in (\psi_1^{-1}(W')S^{-1}). 
$$
Direct calculations show that these elements can be represented as 
$$
\tilde{w}_i P^n = w'_i\sigma (P)^n + \mbox{"lower order terms"}, \mbox{\quad} \tilde{w}_i Q^n = w'_i\sigma (Q)^n + \mbox{"lower order terms"}.
$$
Hence, since $n\gg 0$, we must have $w'_i\sigma (P)^n, w'_i\sigma (Q)^n \in W_0$. Therefore, $w'_i\sigma (P)^n, w'_i\sigma (Q)^n$ must be homogeneous polynomials of orders $q_i+nk$, $q_i+n(l+1)$. Since the characteristic schemes of $P$ and $Q$ have no intersection, this mean that $w'_i$ must be a homogeneous polynomial of order $q_i$. But then since $w'_i \notin W_0$ and due to the property \eqref{extraproperty} the polynomials $w'_i\sigma (P)^n, w'_i\sigma (Q)^n$ will contain a nonzero monomial of type $cz_1^{-a}z_2^{b}\notin W_0$ for $b>0$, a contradiction. Thus, all $\tilde{w}_i$ must be zero, and $W'=W$, where from $CM(\cf )=\cf$. 
\end{proof}

\subsection{Geometric properties of rational commutative algebras of PDOs}

As we have mentioned in Introduction, there are examples  of algebraically integrable commutative rings of PDOs whose affine spectral surface satisfy the following property: its normalisation is $\da^2$. The following theorem clarifies what is the normalisation of the projective spectral surface $X$.  

\begin{theo}
\label{completion_of_plane}
Let $X$ be a projective surface, $C\subset X$ --- an integral Weil divisor not contained in the singular locus of $X$ which is also an ample $\dq$-Cartier divisor and $C^2=1$. Assume that $X\backslash C\simeq \da^2$. 

Then $X\simeq \dpp^2$, $C\simeq \dpp^1$.
\end{theo}

\begin{proof}
Since $C$ is not contained in the singular locus of $X$, we can choose a point $P$ regular on $C$ and on $X$. Choose  local parameters $u,t$ at $P$ such that $t$ is a local equation of $C$ at the point $P$ and $u\in \co_P$ restricted to $C$ is a local equation of the point $P$ on $C$. 

Now we have the natural isomorphism
$$
\pi : \hat{\co}_P \rightarrow k[[u,t]].
$$
Using this isomorphism we can repeat the construction of the subspace $A$ from section \ref{data} and define  $A \eqdef \chi_1(H^0(X\backslash C, \co ))$.  

Repeating the arguments from the proof of \cite[lemma 3.6]{Zhe2} we obtain that for all $n\ge 0$ 
$$
H^0(X,\co_X(nC'))\simeq A_{nd},
$$
where $A_{l}=A\cap t^{-l}k[[u,t]]$. Since $C^2=1$, we must have for all $n\gg 0$ 
\begin{equation}
\label{f1}
\dim (A_{nd}/A_{(n-1)d})=d^2n+const.
\end{equation}

Consider any element $a\in A_{nd}$ such that $\nu (a)=(*,nd)$. We claim that $*\le nd$. 

Indeed, by \cite[sec.3.4]{Ku3} there is a canonically defined ribbon $(C,\ca )$ over the field $k$. Thus, by the proof of \cite[Th.1]{Ku} we can construct the space $\da$ in $k((u))((t))$ which is a generalized Fredholm subspace (see loc. cit. or section \ref{sec.2.3}). As it follows from \eqref{a-property}, the space $\da (nd)$ 
 is naturally isomorphic  to a Fredholm subspace in the field $k((u))$ obtained as the image of the sheaf $\co_X(nC')|_{C}$ under the Krichever map. For $n\gg 0$ we have also $H^0(C, \co_X(nC')|_C)\simeq A_{nd}/A_{nd-1}$ and by \eqref{kriv2} 
\begin{equation}
\label{f2}
\dim (\da (nd)\cap k[[u]])=h^0(C, \co_X(nC')|_C)=nd +const. 
\end{equation}

\begin{nt}
Alternately, we can just repeat the construction of the generalized Krichever map from \cite{Pa} or from \cite{Os} in our situation (replacing a Cartier divisor there with a $\dq$-Cartier one) to avoid referring to the theory of ribbons.
\end{nt}

Since $C^2=1$, we have that the Euler characteristic 
$$
\chi (\da (nd)) =nd + const.
$$
Now we can apply arguments from the proof of \cite[Th.1]{ZO} to show that $*\le nd$. Assume the converse. We have $a\cdot \da (0)\subset \da (nd)$. It is easy to see that $\chi (a\cdot \da (0))=\chi (\da (0))+*$. Now we have
$$
\chi (\da (nd))=nd + const < * + const =\chi (\da (0))+*=\chi (a\cdot \da (0))\le \chi (\da (nd)), 
$$ 
a contradiction. 

Now note that, since $X\backslash C\simeq \da^2$, we have $A\simeq k[p,q]$. So, the space $A$ is generated by monomials $p^kq^l$. Because of the claim and formulas \eqref{f1} and \eqref{f2} we conclude that (without loss of generality) $\nu (p)=(0,1)$, $\nu (q)=(1,1)$ (since any other values would make these formulas impossible). But then $A\simeq k[t^{-1},ut^{-1}]$ and 
$X\simeq \Proj (\oplus A_n)=\dpp^2$, $C\simeq \Proj (\oplus A_n/A_{n-1})=\dpp^1$ (cf. the proof of \cite[lemma 3.3]{Zhe2}).

\end{proof}

\subsubsection{The example of an affine non-spectral surface}

\begin{ex}
\label{counterex_aff}
Using the idea from the proof of theorem \ref{completion_of_plane} we can give an example of an affine surface which can not be a spectral surface of any ring $B$ of PDO's of rank one satisfying the property \ref{property}. For example,  consider the ring  
\begin{equation}
\label{r}
A=k[X_1, X_2, X_3]/(F),
\end{equation}
where $F=X_1X_2 + X_3 +\sum_{q=1}^{r} g_qX_1^q$, and $g_i\in k[X_3]$ are any polynomials and $k$ is an algebraically closed field. 

Then (see \cite[Ch.VII,\S 3, Ex.5]{Bu}) $A$ is a factorial ring, and $\Spec (A)$ is a rational affine surface. It is easy to see that $A$ is not isomorphic to a polynomial ring $k[u,v]$ for generic $g_i$ and that $\Spec (A)$ is a smooth surface. Assume that there exists a ring $B\subset D$ of rank one satisfying the property \ref{property} and such that $B\simeq A$. Without loss of generality we can assume that $B$ is a 1-quasi-elliptic strongly admissible ring. 
Since the rank of the ring is one, the rank of the data is also one by the classification theorem \ref{dannye2}. Then the sheaf $\cf$ is coherent of rank one by \cite[Prop.3.3]{Ku3}. By theorem \ref{CMmodules} the sheaf $\cf$ is Cohen-Macaulay.  Since $\Spec (A)$ is smooth, $\cf$ must be locally free on $\Spec (A)$. But since $A$ is a factorial ring, we have $Cl (A)\simeq \Pic (A) =0$, thus $\cf |_{\Spec (A)}\simeq \co_{\Spec (A)}$. 
But then the space $W$ of the corresponding Schur pair must be equal to the space $A$. Therefore, $A\simeq k[ut^{-1},t^{-1}]$ (where $u,t$ are the parameters from \eqref{psi_1}), a contradiction.   
\end{ex}

\subsubsection{The Darboux transformation for PDOs with rational spectral surface}

\begin{theo}
\label{Darboux}
Let $B\subset D$ be a commutative ring of rank $\Rk (B)=1$ satisfying the properties \eqref{property}. Assume that the normalization of $\Spec (B)$ is isomorphic to $\da^2$. 

Then there exists a PDO $F$ such that $F^{-1}BF\subset k[\partial_1, \partial_2]$. 

More precisely, $F=S\partial_2^n$, where $S$ is an operator as in the analogue of the Sato theorem \ref{Sato}.  
\end{theo} 

\begin{proof}
We can assume without loss of generality that $B$ is a finitely generated 1-quasi-elliptic strongly admissible ring satisfying the property \eqref{extraproperty} (see the beginning of section \ref{CRO} and arguments before remark \ref{remark}), since our assertion don't depend on linear changes of coordinates.

Since the rank of the ring $B$ is one, the rank of the corresponding geometric data $(X,C,P,\cf ,\pi ,\phi )$ is also one by the classification theorem \ref{dannye2}. Then the sheaf $\cf$ is coherent of rank one by \cite[Prop.3.3]{Ku3} and $C$ is a rational curve with $C^2=1$ by \cite[Prop.3.2]{Ku3}. 

The assumption about normalization is equivalent to the assumption that the normalization of $\Spec (B)\simeq  X\backslash C$  is isomorphic to $\da^2$. Note that this assumption is equivalent to 
the assumption that the normalization of $X$ is isomorphic to $\dpp^2$. Indeed, if $p:\dpp^2\simeq \tilde{X}\rightarrow X$ is the normalization morphism, then $p^*(C)$ is a rational irreducible curve. Thus, we have $p^*(C)$ is an ample rational Cartier-Weil divisor on $\dpp^2$ with $p^*(C)^2=1$, i.e. $p^*(C)=\dpp^1$. Therefore, the normalization of $\Spec (B)$ is isomorphic to the complement to $\dpp^1$ in $\dpp^2$, i.e. $\da^2$. 
The converse statement follows from the same arguments together with theorem \ref{completion_of_plane}.

Let $(\tilde{X}\simeq \dpp^2, p^*(C)\simeq \dpp^1, p^*(P)=\tilde{P})$ be the normalization of $(X,C,P)$. Since $P$ is regular, the local rings $\co_{X,P}$ and $\co_{\dpp^2,\tilde{P}}$ are canonically isomorphic. 

Now we can repeat the arguments from the beginning of the proof of theorem \ref{completion_of_plane} to get an embedding of $H^0(\dpp^2\backslash \dpp^1,\co )$ to the same space $k[[u]]((t))$ (here $u,t$ are local parameters at the point $P\in X$). Let's denote the image by $A'$. As we have seen above, $A'$ must be the normalization of $A$. The arguments from the proof of theorem \ref{completion_of_plane} show that in fact $A'\simeq k[p,q]$ with the highest terms of the series $p,q$ equal to $t^{-1}$, $ut^{-1}$ correspondingly  (so, $\Sup (A')=k[ut^{-1},t^{-1}]$).  

Set $A''=\psi_1^{-1}(A')$ (see \eqref{psi_1}). Then we have $\Sup (A'')=k[z_1^{-1},z_2^{-1}]$, and $A''$ is a $1$-space. By \cite[Lemma 2.11, 2),3)]{Zhe2} there is an operator $S$ such that $S^{-1}\partial_1S=\psi_1^{-1}(q)$, $S^{-1}\partial_2S=\psi_1^{-1}(p)$, and $S$ satisfies the condition $A_1$.  Thus, $S\in \Adm_1$. 

Now consider the Schur pair $(\psi_1^{-1}(A),W)$ {\it from theorem} \ref{schurpair} corresponding to the ring $B$. Consider the equivalent pair $(\ca =S\psi_1^{-1}(A)S^{-1}, \cw =WS^{-1})$. Then the ring $SA''S^{-1}=k[z_1^{-1},z_2^{-1}]$ is the normalization of the ring $S\psi_1^{-1}(A)S^{-1}$ (in the field $k(z_1,z_2)\subset k((z_1))((z_2))$). Thus, all elements of the space  $S\psi_1^{-1}(A)S^{-1}$ are polynomials in $z_1^{-1},z_2^{-1}$. 

The space $WS^{-1}$ is a finitely generated module over $S\psi_1^{-1}(A)S^{-1}$. Without loss of generality we can assume that $1\in WS^{-1}$ by taking another equivalent Schur pair $(\ca , \cw T)$ if needed (for an appropriate operator $T$ with constant coefficients; $T$ just changes the trivialisation $\phi$ in definition \ref{geomdata}, item 6). By construction of the Schur pair given in section \ref{data} we have $\cw \subset k(z_1,z_2)$ (as this Schur pair corresponds to a pair coming from the geometric data with an appropriate trivialization $\phi$, and the rank of the coherent sheaf $\cf$ is one).  

So, $\cw$ is generated by a finite number of elements from $k(z_1,z_2)$ over $\ca$. Since $\cw$ is a $1$-space, we can choose the generators to be the elements satisfying the condition $A_1$. 
Let's denote by $Q$ their common denominator. From lemma \ref{technical} (see below) it follows that  $\ord_{\Gamma}(Q)=(0,n)$, where $n=\Ord (Q)$ (here and below we identify $z_1$ with $\partial_1^{-1}$, $z_2$ with $\partial_2^{-1}$; in this case $\Ord (Q)=\deg (Q)$, where $\deg$ means the usual degree of the polynomial $Q$ in two variables).  
Consider the equivalent Schur pair $(\ca ,\cw Q/\partial_2^{\deg (Q)})$ (this is a Schur pair since $Q/\partial_2^{\deg (Q)}$ is a zeroth order operator with constant coefficients with $\ord_{\Gamma}(Q/\partial_2^{\deg (Q)})=(0,0)$  satisfying condition $(A_1)$!). Note that all elements from the space $\cw Q/\partial_2^{\deg (Q)}$ are just polynomials in $\partial_1$, $\partial_2$, $\partial_2^{-1}$ with constant coefficients, and the order of these polynomials with respect to $\partial_2^{-1}$ is less or equal to $\deg (Q)$.  

Then from the proof of theorem \ref{Sato} in \cite{Zhe2} immediately follows that the operator $S$ is a (non-commutative) polynomial in $\partial_2^{-1}$ of degree with respect to $\partial_2^{-1}$ not greater than $\deg (Q)$. By remark  \ref{B_of_rank_r} the ring $S\ca S^{-1}$ is a ring of PDO's. Then $S\in k[[x_1,x_2]][\partial_1]((\partial_2^{-1}))$ (this immediately follows from lemmas 2.9, 2.11 item 2,3 in \cite{Zhe2}).  Thus, we can set $F=S\partial_2^{\deg (Q)}$. 

\end{proof}

\begin{lemma}
\label{technical}
Assume that the Laurent expansion of the element 
$$P/Q\in k(\partial_1,\partial_2)\subset k((\partial_1^{-1}))((\partial_2^{-1})), 
$$
where $P,Q\in k[\partial_1,\partial_2]$ are relatively prime, belongs to $k[\partial_1]((\partial_2^{-1}))$. Assume that this expansion satisfies the condition $A_1$. 

Then $\ord_{\Gamma}(Q)=(0,\Ord (Q))$. 
\end{lemma} 

\begin{proof} The proof of this lemma is based on several technical routine elementary calculations, and we will hardly use some technical lemmas from \cite{Zhe2}. 

Assume the converse. Then $Q$ can be represented as a polynomial in $\partial_2$ of the order with respect to $\partial_2$ less than $\Ord (Q)$, say 
$$
Q=q_n\partial_2^n-\sum_{l=0}^{n-1}q_l\partial_2^l, \mbox{\quad } n < \Ord (Q),
$$
where $q_l\in k[\partial_1]$.  Let 
$$
P=\sum_{l=0}^mp_l\partial_2^l. 
$$
Now we will prove our lemma in several steps. 

{\it Step 1.} First we claim that $\deg (q_n)+n=\Ord (Q)$.

Clearly, we always have $\deg (q_n)+n\le\Ord (Q)$. Assume that $\deg (q_n)+n< \Ord (Q)$. Let's show that this will contradict to the condition $A_1$ for the element $P/Q$. Since we are working with series in the field $k((\partial_1^{-1}))((\partial_2^{-1}))$ of pseudo-differential operators with constant coefficients, we can literally  repeat the proofs of lemma 2.8 and corollary 2.1 in \cite{Zhe2} to show that the statements from these claims remain true also for operators from $k((\partial_1^{-1}))((\partial_2^{-1}))$. 

In particular, $Q^{-1}$ don't satisfy the condition $A_1$. Indeed, assume that $Q^{-1}$ satisfies the condition $A_1$. Then $Q^{-1}=q_n^{-1}\partial_2^{-1}Q'$, where $Q'$ is an operator of the form from \cite[Corol.2.1]{Zhe2} satisfying the condition $A_1$ (by \cite[Lemma 2.8]{Zhe2}). Then $Q=(Q^{-1})^{-1}=q_n\partial_2(Q')^{-1}$ must satisfy the condition $A_1$ by \cite[Lemma 2.8, Corol.2.1]{Zhe2}. But $Q$ don't satisfy the condition $A_1$ by our assumption (namely, the term with the first coefficient $q_i$ of $Q$ such that $\deg (q_i) + i=\Ord (Q)$ will contradict the condition $A_1$), a contradiction. 

Let $P=P_1+P_2$ be any decomposition of $P$ in a sum of two PDOs with constant coefficients such that $P_1$ satisfies the condition $A_1$ and the degree of $P_2$ with respect to $\partial_2$ is less than $m$ ($P_2$ may be zero). Let $Q^{-1}=Q_1+Q_2$ be any decomposition of $Q^{-1}$ in a sum of two pseudo-differential operators from $k((\partial_1^{-1}))((\partial_2^{-1}))$
such that $Q_1$ satisfies the condition $A_1$ and the degree of $Q_2$ with respect to $\partial_2$ is less than $-n$ (since $Q^{-1}$ don't satisfy the condition $A_1$, $Q_2$ is not zero). Denote by $\alpha$ the first coefficient of $Q_2$, and by $\beta$ the first coefficient of $P_2$ if $P_2\neq 0$.  
Now we have two cases: if $P_2=0$ then the product $PQ^{-1}$ will not satisfy the condition $A_1$, because the coefficient of $PQ^{-1}$ containing  $p_m\alpha $ will not satisfy it; if $P_2\neq 0$ then the product $PQ^{-1}$ will not satisfy the condition $A_1$, because the coefficient of $PQ^{-1}$ containing $\beta\alpha$ will not satisfy it. Thus, $PQ^{-1}$ does not satisfy the condition $A_1$, a contradiction. 

{\it Step 2.} Now the idea of the proof is to come to a contradiction with the assumption that $q_n\neq const$.

 Obviously, we can multiply the element $P/Q$ on an appropriate degree of $\partial_2^{-1}$ to make the degree of its Laurent expansion to be zero. Thus, we can assume without loss of generality that $P,Q$ are polynomials in $\partial_2^{-1}$ with nonzero free terms $p_m$, $q_n$ correspondingly. 
 
Now  we can write
\begin{equation}
\label{p/q}
P/Q=(\sum_{l=0}^m{p_l}\partial_2^{l-m})(\sum_{l=0}^n{q_l}\partial_2^{l-n})^{-1}=
(\sum_{l=0}^m\frac{p_l}{q_n}\partial_2^{l-m})(\sum_{i=0}^{\infty}(\sum_{l=0}^{n-1}\frac{q_l}{q_n}\partial_2^{l-n})^i).
\end{equation} 
Note that not all $q_{i}$ are divisible by $q_n$. Indeed, otherwise $(q_n^{-1}Q)\in k[\partial_1,\partial_2]$ and therefore 
$$
q_n^{-1}P=(PQ^{-1})(q_n^{-1}Q)\in k[\partial_1]((\partial_2^{-1}))\cap k((\partial_1^{-1}))[\partial_2]=k[\partial_1,\partial_2]
$$ 
i.e. $P$ and $Q$ are divisible by $q_n\neq const$, a contradiction. 

Note that we can reduce the proof to the case $\deg (P)\le n-1$ (the degree now means the degree with respect to $\partial_2^{-1}$). Indeed, it is easy to see that $p_m$ must be divisible by $q_n$. Since $P/Q \in k[\partial_1]((\partial_2^{-1}))$, all expressions of type $(P/Q - a)\partial_2^k$ will again belong to $k[\partial_1]((\partial_2^{-1}))$ for any polynomial $a\in k[\partial_1]$. Thus, if we take $a=p_m/q_n$, then $(P/Q - a)\partial_2=P'/Q$, where $\deg (P') <\deg (P)$ if $m\ge n$. Note that $P'\neq 0$, since $P,Q$ are relatively prime. 

Analogously we can reduce the proof to the case $\deg (P)=0$. Indeed, using Euclidean algorithm, we can always find polynomials $a\in k[\partial_1]$ and $F\in k[\partial_1,\partial_2^{-1}]$ such that $\deg (aQ-FP)<\deg (P)$ if $\deg (P)\neq 0$. Again $(aQ-FP)\neq 0$, since $P,Q$ are relatively prime and $\deg (P)\neq 0$. Thus, $F(P/Q)-a=P'/Q$ with $\deg (P')<\deg (P)$, $P'\neq 0$.  

At last, in the case $\deg (P)=0$ the proof follows immediately from \eqref{p/q}: $P$ must be divisible by infinite power of some prime factor of $q_n$, i.e. $P=0$, a contradiction. 

\end{proof}

\section{Examples} 

In this section we give several examples. 

\subsubsection{"Trivial" commutative algebras of operators} 
First we would like to explain the geometric picture for a class of "trivial" examples. These are examples of rings of commuting operators in $\hat{D}$ containing, say, the operator $\partial_1$. In this case all operators obviously don't depend on $x_1$. Nevertheless, the geometry of the corresponding surfaces and even the naive moduli space of sheaves from geometric data are non-trivial. 

Note that, if we have a commutative $1$-quasi-elliptic strongly admissible ring $B\subset D$ satisfying properties \eqref{property}, \eqref{extraproperty} and containing the operator $\partial_2$, then after a linear change $\partial_1\leftrightarrow \partial_2$ the ring $B$ will remain $1$-quasi-elliptic strongly admissible and will contain the operator $\partial_1$. So, in particular, the well known example of the quantum Calogero-Moser system (see \cite{OP},  \cite[sec. 5.3]{BEG} and \cite[Ex.4.3]{Zhe2}) belongs to this class of "trivial" examples. We would like to emphasize that in \cite[sec. 5.3]{BEG} the affine spectral surface of this system was calculated: it is $\da^1\times H$, where $H$ is some hyperelliptic curve. So, by \cite[Th.2.1]{Ku3} the projective spectral surface $X$ from the corresponding geometric data is normal, and singularities appear only on the curve $C$ (which is rational).  

\begin{theo}
\label{trivial}
Let $B\subset \hat{D}$ be a Cohen-Macaulay finitely generated 1-quasi-elliptic strongly admissible ring of commuting operators (note that by \cite[Th.4.1]{Zhe2} any finitely generated 1-quasi-elliptic strongly admissible ring $B$ lies in a Cohen-Macaulay finitely generated 1-quasi-elliptic strongly admissible ring). 

Then $B$ contains $\partial_1$ if and only if the divisor $C$ of the corresponding geometric data is Cartier, the sheaf $\cf$ is coherent of rank one, $\co_C(C)\simeq \co_C(P)$ and the map 
$$H^1(X,\co_X)\rightarrow H^1(X,\co_X(C))$$ 
is injective. 

Moreover, if the ground field $k$ is uncountable and algebraically closed, the sheaf $\cf$ is Cohen-Macaulay on $X$. 
\end{theo}

\begin{proof} Recall that the surface $X$ corresponding to $B$ is Cohen-Macaulay by \cite[Th.3.2]{Ku3}. 

Assume first that $B$ contains $\partial_1$. Since $B$ is 1-quasi-elliptic strongly admissible, this means that $\rk (B)=1$. Also this means that for any $n\gg 0$ there are operators $P_n\in B$ with $\ord_{\Gamma}(P_n)=(0,n)$. 
Thus, we can give an approximation of the dimension of the space $A_n\subset A$ (where, as usual, $A$ means the space of the Schur pair corresponding to the ring $B$): $\dim_k (A_n)\sim n^2/2$ for all $n\gg 0$. Then it follows from the asymptotic Riemann-Roch formula \eqref{rrf} that $C^2=1$ (since $A_{nd}\simeq H^0(X,\co_X(ndC))$ for all $n\gg 0$, see section \ref{data}). Since $\rk (B)=1$, the rank of the corresponding geometric data is also one by theorem \ref{dannye2}. Thus, by \cite[Prop.3.2]{Ku3} the corresponding sheaf $\cf$ is coherent of rank one. 

Now let's prove that $C$ is a Cartier divisor. Our arguments will be very similar to the arguments from the proofs of lemma 3.3 in \cite{Zhe2} or theorem 2.1 in \cite{Ku3}. 
Recall that $X\simeq \Proj \, \tilde{A}$ and the divisor $C$ is defined by the homogeneous ideal $I=(s)$. It is not contained in the singular locus, since it contains the regular point $P$. Since $\tilde{A}$ is a finitely generated $k$-algebra with $\tilde{A}_0=k$, by~\cite[Ch.III, \S~1.3, prop.~3]{Bu} there exists an integer $d \ge 1$ such that the
$k$-algebra $\tilde{A}^{(d)}=\bigoplus\limits_{k=0}^{\infty} \tilde{A}_{kd}$ is finitely generated by elements from $\tilde{A}_1^{(d)}$ as a $k$-algebra (here $\tilde{A}_1^{(d)}= \tilde{A}_d$). 

Let's show that the divisor $dC$ is an effective Cartier divisor. 
We consider the subscheme $C'$ in $X$ which is  defined by the homogeneous ideal $I^d=(s^d)$ of the ring $\tilde{A}$.  The topological space of the subscheme $C'$ coincides with the topological space of the subscheme $C$ (as it can be seen on an affine covering of $X$). The local ring $\co_{X,C}$ coincides with the valuation ring of the discrete valuation on $\Quot (A)$ induced by the discrete valuation $\nu_t$ on the field $k((u))((t))$:
$$
\co_{X,C} = \tilde{A}_{(I)}= \{  a s^n / b s^n  \, \mid \,  n \ge 0,  a \in A_n, \,  b \in A_n \setminus A_{n-1} \} \mbox{.}
$$
The ideal $I$ induces the maximal ideal in the ring $\co_{X,C}$, and the ideal $I^d$ induces the $d$-th power of the maximal ideal. Therefore, if we will prove that the ideal $I^d$ defines an effective Cartier divisor on $X$, then the cycle map on this divisor is equal to $dC$ (see \cite[Appendix~A]{Ku3}). 
By~\cite[prop. 2.4.7]{EGAII} we have $X=\Proj \, \tilde{A}\simeq \Proj \, \tilde{A}^{(d)}$. Under this isomorphism the subscheme $C'$ is defined by the homogeneous ideal $I^d \cap \tilde{A}^{(d)}$ in the ring $\tilde{A}^{(d)}$. This ideal is generated by the element $s^d\in \tilde{A}_1^{(d)}$.
The open affine subsets $D_+(x_i)=\Spec \, \tilde{A}^{(d)}_{(x_i)}$ with $x_i\in \tilde{A}_1^{(d)}$ define a covering of $\Proj \, \tilde{A}^{(d)}$. In every ring $\tilde{A}^{(d)}_{(x_i)}$ the ideal $(I^{d} \cap \tilde{A}^{(d)})_{(x_i)}$ is generated by the element $s^d/x_i$. Therefore the homogeneous ideal $I^{d} \cap \tilde{A}^{(d)}$ defines an effective Cartier divisor. 

Now let's show that the $k$-algebra $\tilde{A}^{(m)}$ is finitely generated by elements from $\tilde{A}^{(m)}_1$ for all $m\gg 0$. By theorem \ref{schurpair} it is equivalent to show that the $k$-algebra $\tilde{B}^{(m)}$ is finitely generated by elements from $\tilde{B}^{(m)}_1$ for all $m\gg 0$.

Let $d$ be such a number that all generators of ${B}$ lie in $B_d$ and for all $n\ge d$ there are elements $P_n\in B$ with $\ord_{\Gamma}(P_n)=(0,n)$ (the same will be true also for the ring $A$). Let $\Sigma$ denote the set of all numbers $a\in \dz_+$ such that there are operators $Q$ in $B_d$ with $\ord_{\Gamma}(Q)=(*,a)$. Since $\partial_1\in B$, we have that for any $m>d$ and any $a\in \Sigma$ there are operators $Q$ in $B$ with $\ord_{\Gamma}(Q)=(m,a)$. 

Now let $m>2d$ be any number such that $\tilde{B}^{(m)}$ is finitely generated by elements from $\tilde{B}^{(m)}_1$. It suffices to show that  $\tilde{B}^{(m+1)}$ is also finitely generated by elements from $\tilde{B}^{(m+1)}_1$. To show this it suffices to show that any element from $\tilde{B}^{(m+1)}_k$ can be represented as a sum of products of  elements from $\tilde{B}^{(m+1)}_{k-1}$ and from $\tilde{B}^{(m+1)}_1$. 
The space $\tilde{B}^{(m+1)}_1$ has two special operators: $Q_1=\partial_1^{m+1}$ and $Q_2$ with $\ord_{\Gamma}(Q_2)=(0,m+1)$. 
As it follows from what we have said above, for any $l\ge 2d$ and any $i,j\in \dz_+$ such that $i+j=l$ there is an element $Q\in B$ with $\ord_{\Gamma}(Q)=(i,j)$. Thus, any element $Q\in \tilde{B}^{(m+1)}_k$ can be written as a sum of an element with the order less than $\ord_{\Gamma}(Q)$ and a product of an element from $\tilde{B}^{(m+1)}_{k-1}$ and $Q_1$ or $Q_2$. By induction we obtain our claim. 

Now the arguments above for $mC$ and $(m+1)C$ (instead of $dC$) show that $mC$, $(m+1)C$ are Cartier divisors. But then $C$ must be also a Cartier divisor. 

Now let $(\da ,\dw )$ be the pair from $k((u))((t))$ corresponding to our geometric data (see section \ref{data}). 
As it follows from section \ref{data} (namely, from \eqref{a-property}, \eqref{1-KP} and \eqref{intersection})
$$
\da (n) \cap k[[u]]\simeq A_n/A_{n-1} \simeq H^0(C,\co_C(nC)), \mbox{\quad} 
$$ 
for all $n\gg 0$ and $\da (n)$ is the image of $(C,P,\co_C(nC),u,id)$ under the Krichever map (cf. also \eqref{car-property}). From one-dimensional KP theory (see \eqref{kriv1}, \eqref{kriv3}) we have then that $\da (n)\cdot u^{-n}$ is the image of the quintet $(C,P,\co_C(nC)(-nP),u,id)$ under the Krichever map. Since $\partial_1^n\in B_n\backslash B_{n-1}$, we have that $t^{-n}u^n\in A_n/A_{n-1}$. Hence, 
$$
H^0(C,\co_C(nC)(-nP))\simeq \da (n)\cdot u^{-n}\cap k[[u]]\simeq k,
$$
and by Riemann-Roch $h^1(C,\co_C(nC)(-nP))=g_a(C)$. But then $\co_C(nC)(-nP)\simeq \co_C$ for all $n\gg 0$, i.e. $\co_C(C)\simeq \co_C(P)$.  

Now we have two possibilities for the number $h^0(C,\co_C(C))$: it is either 1 or 2. If it is equal to 1, then this means that
\begin{equation}
\label{konec}
H^0(C,\co_C(C))\simeq \da (1)\cap k[[u]]\simeq A_1/A_0,
\end{equation}
because $\partial_1\in B_1\backslash B_0$ and we have always $A_1/A_0 \subset \da (1)$. Note that we always have the embeddings
$$
\da \cap (k[[u]]((t))+k((u))[[t]]) \overset{\cdot t}\hookrightarrow \da \cdot t\cap (k[[u]]((t))+k((u))[[t]]),
$$
$$
\da \cap k((u))[[t]] \overset{\cdot t}\hookrightarrow \da \cdot t \cap k((u))[[t]],
$$
$$
\da \cap k[[u]]((t)) \simeq \da \cdot t \cap k[[u]]((t)).
$$
Thus, we have a natural linear map 
\begin{equation}
\label{embedding}
\frac{\da \cap (k[[u]]((t))+k((u))[[t]])}{(\da \cap k((u))[[t]])+(\da \cap k[[u]]((t)))} \longrightarrow 
\frac{\da \cdot t \cap (k[[u]]((t))+k((u))[[t]])}{(\da \cdot t \cap k((u))[[t]])+(\da \cdot t \cap k[[u]]((t)))}. 
\end{equation}
By \eqref{picture_cohomology1} this map coincides with the map $H^1(X,\co_X)\rightarrow H^1(X,\co_X(C))$. Let's show that the kernel of this map may contain only elements from 
\begin{equation}
\label{kernel}
(\da_1 \cap (k[[u]]((t))+k((u))[[t]]))+[(\da \cap k((u))[[t]])+(\da \cap k[[u]]((t)))],
\end{equation}
where $\da_1=\da\cap t^{-1}\cdot k((u))[[t]]$. 
From this and \eqref{konec} it follows that the map $H^1(X,\co_X)\rightarrow H^1(X,\co_X(C))$ is injective. 
Let $a\in \da \cap (k[[u]]((t))+k((u))[[t]])$ be a lift of an element from the kernel. Then $a\cdot t=a_1+a_2$, where $a_1\in (\da \cdot t \cap k((u))[[t]])$ and $a_2\in (\da \cdot t \cap k[[u]]((t)))$. Since $a_1t^{-1}\in (\da \cap k((u))[[t]])$, we have 
$$
a_2t^{-1}=a-a_1t^{-1}\in \da \cap (k[[u]]((t))+k((u))[[t]]).
$$
But $a_2t^{-1}\in \da_1$ and also gives an element of the kernel. 

Now assume that $h^0(C,\co_C(C))=2$. 
This means that the image of the sheaf $\co_C$ in $k((u))$ under the Krichever map contains an element of order $-1$. Hence, this image is isomorphic to the ring $k[u^{-1}]$, i.e. $C\simeq \dpp^1$. But then the surface $X$ must be smooth along $C$, hence $X$ must be normal since $X$ is Cohen-Macaulay  and $C$ is an ample divisor. 
Then by \cite[Th.2.5.19]{Bad} and \cite[Corol.2.5.20]{Bad} there is an open neighbourhood of $C$ isomorphic to an open neighbourhood of a line in $\dpp^2$. Since $\zeta (\cf )$ is a torsion free sheaf and $h^0(C,\zeta (\cf (nC)))=\dim W_n/W_{n-1} =n+1$ for all $n\gg 0$, we have $\zeta (\cf )\simeq \co_C$. Since $\cf$ is Cohen-Macaulay, it is locally free on the smooth open neighbourhood of $C$. Since $Cl (\dpp^2)=Cl (\dpp^2\backslash Z)\simeq \dz$ for any closed subscheme of codimension  greater than one, we must have $\cf\simeq \co_X$ on this open set (since otherwise its restriction on $C=\dpp^1$ would be not trivial).  But then (e.g. by \cite[Prop.1.1.6]{Bad}) $W_n\simeq H^0(X, \cf (nC))\simeq H^0(X, \co_X (nC))\simeq A_n$ since $X$ and $\cf$ are Cohen-Macaulay. Thus, $A\simeq k[a,b]$ and $X=\dpp^2$. Hence, $H^1(X, \co_X)=0$ and we are done. 

At last, from formulas \eqref{kriv1} and \eqref{kriv2} one can easily deduce that the sheaf $\cf$ fulfils the assumptions of proposition \ref{zero_cohomology}. Hence it is Cohen-Macaulay on $X$. 

\smallskip 

Conversely, assume that $C$ is a Cartier divisor, $\cf$ is a coherent sheaf of rank one, the map $H^1(X,\co_X)\rightarrow H^1(X,\co_X(C))$ is injective, $\co_C(C)\simeq \co_C(P)$. Then by \cite[Rem.3.3]{Ku3} the rank of the data is one. 
As we have seen above, the condition on cohomology means that  the kernel of the cohomology map \eqref{embedding} is zero. This means (see \eqref{kernel}) that all elements from $\da_1 \cap (k[[u]]((t))+k((u))[[t]])$ can be represented as a sum of an element from $\da_1\cap k[[u]]((t))$ and an element from $\da_1\cap k((u))[[t]]$. In particular, for any element from $\da (1)\cap k[[u]]$ there exists an element from $\da_1\cap k[[u]]((t))$ with the same support (multiplied by $t^{-1}$). This means that 
$$
H^0(C,\co_C(C))\simeq H^0(C,\co_C(P))\simeq \da (1)\cap k[[u]]\simeq A_1/A_0
$$
(since the rank of the data is one). 
Note that $\da (1)$ contains an element of order one (since $\da (1)$ is the image of $\co_C(P)$ under the Krichever map). Thus, there is an element in $A_1$ with the least term $ut^{-1}$. But this element will give us the operator  $\partial_1$ after applying the map $\psi_1^{-1}$ and conjugating by the Sato operator from theorem \ref{Sato}. 
\end{proof}

\subsection{Some examples}

\begin{ex}
\label{ex1}
This is an example of a surface, divisor and point for which we can calculate all possible geometric data of rank one, corresponding Schur pairs and corresponding algebras of commuting operators. More precisely, we start from a ring $A$, and describe all possible Schur pairs with the ring $A$ as a stabiliser ring. This description is possible due to using concrete formulas from the classical KP theory in dimension one; these formulas also lead to a precise description  of commuting operators. 
Notably, we will see that the map $\zeta$ restricted to the set of all sheaves from these geometric data maps this set surjectively to the dense open subset of the compactified generalized jacobian of the curve $C$ consisting of sheaves with trivial cohomologies. We will see also that for this surface there are no other rings of commuting PDOs except one ring of operators with constant coefficients. 

Consider the ring 
$$
A=k\langle \partial_2^2, \partial_2(\partial_2^2+3\partial_1^2), \partial_1\rangle \subset k[\partial_1, \partial_2].
$$ 
It is easy to see that $A\simeq k[h][z,x]/(z^2-x(x+3h^2)^2)$ (where $\partial_2(\partial_2^2+3\partial_1^2) \mapsto z$, $\partial_2^2\mapsto x$, $\partial_1\mapsto h$) and that $F=k[\partial_1,\partial_2]$, where $F$ denote the normalization of $A$. It is also clear that $A$ is a 1-quasi-elliptic strongly admissible ring. 

Recall that, having such a ring $A$, we can construct a part of geometric data, namely the surface $X$, the divisor $C$ and the point $P$ (see section \ref{data}).   This part can be described in a more explicit way: we have the embedding
$$
\tilde{A}\simeq k\langle \partial_2^2, \partial_2(\partial_2^2+3\partial_1^2), \partial_1, T\rangle \subset \tilde{F}\simeq k[\partial_2, \partial_1,T],
$$
which induces the normalisation morphism $\pi :\Proj (\tilde{F})\rightarrow \Proj (\tilde{A})$, and $X=\Proj (\tilde{A})$ can be considered as a  subscheme in the weighted projective space $\Proj (k[x,z,h,T])$, where weights of $(x,z,h,T)$ are $(2,3,1,1)$. Thus, $\partial_2=z/(x+3h^2)=x(x+3h^2)/z$ where from 
$$
\pi_*\co_{\sdp^2}=\co_X+\co_X(-1)\partial_2  
$$
and
$$
\pi_*\co_{\sdp^2}/\co_X \simeq \co_X(-1)\partial_2/\co_X(-1)\partial_2\cap \co_X\simeq \co_X(-1)/\co_X(-1)\cap \co_X\frac{1}{\partial_2}
$$
and 
$$
\co_E=\co_X/\co_X\cap \co_X(1)\frac{1}{\partial_2} \simeq \co_X/\co_X(-3)z+\co_X(-2)(x+3h^2),
$$
where $E$ is the singular locus of $X$ (cf. example 3.3 in \cite{Zhe2}). So, $E=\Proj (k[h,T])=\dpp^1$ and $\pi_*\co_{\sdp^2}/\co_X\simeq \co_E(-1)$, where from $H^1(X,\co_X)=0$. 

Let's note that, if we have a geometric data $(X,C,P,\cf ,...)$ where $X,C,P$ are defined by the ring $A$ and the sheaf $\cf$ is coherent of rank 1, then the corresponding Schur pair $(A,W)$ induces a 1-dimensional Schur pair $(A',W')$, where 
$$
A'=\overline{k((\partial_1))} \langle \partial_2^2, \partial_2(\partial_2^2+3\partial_1^2)\rangle ,
$$
and $W'$ is a space over $K=\overline{k((\partial_1))}$ generated by elements  from $W$ (thus, $A',W'\subset K((\partial_2^{-1}))$). The Schur pair $(A',W')$ corresponds to a one-dimensional geometric quintet $(C',P',\cf', \ldots )$ (see \cite[Th.4.6]{Mu} or \cite{Mul}, see also section \ref{svva}), where $C'$ is the nodal curve over $K$ and $\cf'$ is a torsion free rank one sheaf on $C'$ with $H^0(C',\cf ')=H^1(C', \cf ')=0$. It is not difficult to see that the divisor $C$ on the surface $X$ is naturally isomorphic to the nodal curve too, whose affine equation (the equation of $C\backslash P$) is $\tilde{y}^2=y(y+3)^2$. 
  
On the other hand, all torsion free rank one sheaves on this nodal curve (cf. \cite{Rego}) as well as corresponding Schur pairs of one-dimensional geometric data can be explicitly described as follows (cf. \cite[Sec 3]{SW}). The nodal curve $C'$ can be thought of as a projective line with two glued points, whose local coordinates are $a$ and $-a$ (with respect to the local coordinate $z$ on $\dpp^1\backslash P'$). It is not difficult to see that in our case 
$$
a=i\sqrt{3}\partial_1,
$$ 
and for the curve $C$ it is equal to $i\sqrt{3}$. 
Now we can use the well known formula of Baker-Akhieser function associated to a line bundle on the curve to describe the corresponding spaces of Schur pairs. Recall that the Baker-Akhieser function can be written in the form 
$\psi (x,z)\exp{(xz^{-1})}=S(x,\partial^{-1})(\exp{(xz^{-1})})$ (where $z$ is a local parameter at a point on the curve).

For the unique non-locally free sheaf $n_*(\co_{\dpp^1})$ of degree zero (where $n:\dpp^1 \rightarrow C'$ is the normalization map) the corresponding space $W'$ is equal to $K[\partial_2]$. This space comes from $W=k[\partial_1,\partial_2]$, and the pair $(A,W)$ obviously corresponds to the ring $A$ of differential operators with constant coefficients. 

The only locally free sheaf of degree zero which has non-zero cohomologies is $\co_{C'}$. 
For any locally free sheaf $\cl$ parametrized (in the moduli space) by an element $\lambda\in K^*\simeq \Pic (C')$, $\lambda\neq -1$ ($\lambda =-1$ corresponds to $\co_{C'}$) the corresponding space $W'$ is equal to 
$$
K[\partial_2]\cdot S, \mbox{\quad where $S=(1+w\partial_2^{-1})$ and } 
$$
$$
w=-a\frac{\lambda \exp{(x_2a)}-\exp{(-x_2a)}}{\lambda \exp{(x_2a)}+\exp{(-x_2a)}}.
$$

Now we can describe those one-dimensional Schur pairs $(A',W')$ (over $K$) that are induced by two dimensional Schur pairs $(A,W)$ (over $k$). It is easy to see that necessary and sufficient conditions for describing such pairs are the following: all elements from the admissible basis in $W'$ must belong to $k[\partial_1]((\partial_2^{-1}))$ and satisfy the condition $(A_{1})$. Since $A\subset A'$ and all elements from $A$ satisfy the condition $(A_{1})$, it is enough to check this property only for the two first elements from the admissible basis of $W'$. These  elements are
$$
w_0=S|_{x=0}, \mbox{\quad} w_1=\partial_2+\partial_2(w)|_{x=0}\partial_2^{-1}-(w|_{x=0})^2\partial_2^{-1}.  
$$
Thus we must have
$$
-a\frac{\lambda -1}{\lambda +1}=P(\partial_1), \mbox{\quad } -a^2\frac{4\lambda}{(\lambda +1)^2}=Q(\partial_1), 
$$
where $P,Q$ are  polynomials in $\partial_1$ with coefficients in $k$ of degree not higher than 1 and 2 correspondingly. Hence, from the first equality we have 
$$
\lambda =\frac{a-P}{a+P}\in k(\partial_1),
$$
and the second equality  holds for any such $\lambda$ for any such $P$. The same formulas show (due to theorem \ref{trivial}) that for all $-1\neq \lambda \in k^*$  the sheaf $\zeta (\cf )$ (which is defined by the space $\oplus W_{i+1}/W_i$) is a line bundle on $C$ corresponding to $\lambda$. Clearly, $\zeta (\pi_*(\co_{\sdp^2}))\simeq n_*(\co_{\sdp^1})$. Thus, the map $\zeta$ mentioned in the beginning of this example is indeed surjective. 

On the other hand, for any such $\lambda$ we can calculate operators from the corresponding ring of operators. In particular, there will be an operator of the form 
$$
S^{-1}\partial_2^2S =\partial_2^2+2 \partial_2(w)=\partial_2^2-\frac{8a^2\lambda}{(\lambda \exp{(x_2a)}+\exp{(-x_2a)})^2}.
$$
The last term of this operator can not be a polynomial in $\partial_1$, because the exponential function can not belong to an algebraic extension of the field of rational functions. So, by remark \ref{B_of_rank_r} there are no rings of PDOs with the projective spectral surface $X$  except the ring $A$ of operators with constant coefficients. 
\end{ex}

\begin{ex}
\label{ex2}
This is another example of a surface, divisor and point for which we can calculate all possible geometric data of rank one, corresponding Schur pairs and corresponding algebras of commuting operators. 
The map $\zeta$ mentioned in the previous example will be again surjective. But we will see  that for this surface there are many commutative rings of PDOs.

Consider the ring 
$$
A=k\langle \partial_2^2, \partial_2^3, \partial_1 \rangle \subset k[\partial_1, \partial_2]
$$
It is easy to see that $A\simeq k[h][z,x]/(z^2-x^3)$ (where $\partial_2^3 \mapsto z$, $\partial_2^2\mapsto x$, $\partial_1\mapsto h$) and that $F=k[\partial_1,\partial_2]$, where $F$ denotes the normalization of $A$. It is also clear that $A$ is a 1-quasi-elliptic strongly admissible ring. 

Using similar arguments from the previous example one can show that $X$ can be obtained from $\dpp^2$ by glueing one doubled projective line (cf. \cite[Sec. 3.6]{Ku3}). Thus, we have also $H^1(X,\co_X)=0$. Again as in the previous example $X$ is a cone over $C$ which is a cuspidal curve. So, we can use in this case the same ideas and notation. 

Now any Schur pair $(A,W)$ induces a $1$-dimensional Schur pair $(A',W')$ over $K$, where 
$$
A'=\overline{k((\partial_1))}\langle \partial_2^2, \partial_2^3 \rangle .
$$
For the unique non-locally free sheaf $n_*(\co_{\dpp^1})$ of degree zero the corresponding space $W'$ is equal to $K[\partial_2]$. This space comes from $W=k[\partial_1,\partial_2]$, and the pair $(A,W)$ obviously corresponds to the ring $A$ of differential operators with constant coefficients. 

The only locally free sheaf of degree zero which has non-zero cohomologies is $\co_{C'}$. 
For any locally free sheaf $\cl$ parametrized by an element $\lambda\in K\simeq \Pic (C')$, $\lambda\neq 0$ ($\lambda =0$ corresponds to $\co_{C'}$) the corresponding space $W'$ is equal to 
$$
K[\partial_2]\cdot S, \mbox{\quad where $S=(1+w\partial_2^{-1})$ and } 
$$
$$
w=\frac{1}{\lambda -x_2} .
$$
Now $w_0=S|_{x=0}=1+(1/\lambda )\partial_2^{-1}$. To find those pairs $(A',W')$ that are induced by pairs $(A,W)$ we again must have $1/\lambda =P(\partial_1)$ for some linear polynomial $P$. It is not difficult to see that for all such $\lambda$ the spaces $W'$ are induced by $W$ and that the map $\zeta$ is surjective. 

The rings of commuting operators will contain two operators: $\partial_1$ and 
$$
S^{-1}\partial_2^2S=\partial_2^2+\frac{2P(\partial_1)^2}{(1 -x_2P(\partial_1))^2} .
$$
By remark \ref{B_of_rank_r} and by proposition \ref{purity} such a ring is a ring of PDO if and only if $P(\partial_1)$ is a constant. Clearly the sheaves corresponding to such rings are the preimages of the sheaf $n_*(\co_{\sdp^1})$. 

\end{ex}

\noindent H. Kurke,  Humboldt University of Berlin, department of
mathematics, faculty of mathematics and natural sciences II, Unter
den Linden 6, D-10099, Berlin, Germany \\ \noindent\ e-mail:
$kurke@mathematik.hu-berlin.de$

\vspace{0.5cm}

\noindent A. Zheglov,  Lomonosov Moscow State  University, faculty
of mechanics and mathematics, department of differential geometry
and applications, Leninskie gory, GSP, Moscow, \nopagebreak 119899,
Russia
\\ \noindent e-mail
 $azheglov@mech.math.msu.su$

\end{document}